\documentclass[hidelinks]{article}
\date\today

\usepackage{stmaryrd}
\usepackage{amsfonts}
\usepackage{amssymb}
\usepackage{amsmath}
\usepackage{amsthm}
\usepackage{commath,stackrel,enumerate,mathtools,bbold,xcolor,thm-restate,comment, geometry}
\usepackage{enumitem}
\usepackage{amscd}
\usepackage[pdfencoding=auto, psdextra]{hyperref}
\usepackage{tikz}
\usepackage{pgf}
\usepackage{graphicx}

\theoremstyle{definition}
\newtheorem{example}{Example}[section]
\theoremstyle{plain}
\newtheorem{theorem}{Theorem}
\newtheorem{fact}{Fact}

\newtheorem*{claim*}{Claim}
\newtheorem{lemma}[example]{Lemma}

\newtheorem{proposition}[example]{Proposition}

\newtheorem{corollary}[example]{Corollary}

\newtheorem{question}[example]{Question}
\newtheorem*{answer}{Answer}
\theoremstyle{remark}
\newtheorem{remark}[example]{Remark}
\theoremstyle{definition}
\newtheorem{definition}[example]{Definition}
\renewcommand\labelenumi{(\roman{enumi})}
\renewcommand\theenumi\labelenumi

\title{Infinite Hat Problems and Large Cardinals}
\author{Andreas Lietz, Jeroen Winkel}
\date{\today}

\begin{document}
\maketitle
\begin{abstract}
Picture countably many logicians all wearing a hat in one of $\kappa$-many colours. They each get to look at finitely many other hats and afterwards make finitely many guesses for their own hat's colour. For which $\kappa$ can the logicians guarantee that at least one of them guesses correctly? This will be the archetypical hat problem we analyse and solve here. We generalise this by varying the amount of logicians as well as the number of allowed guesses and describe exactly for which combinations the logicians have a winning strategy.

We also solve these hat problems under the additional restriction that their vision is restrained in terms of a partial order. Picture e.g.~countably many logicians standing on the real number line and each logician is only allowed to look at finitely many others in front of them.

In many cases, the least $\kappa$ for which the logicians start losing can be described by an instance of the free subset property which in turn is connected to large cardinals. In particular, $\mathrm{ZFC}$ can sometimes not decide whether or not the logicians can win for every possible set of colours.
\end{abstract}
\tableofcontents
\section{Introduction and Main Results}\label{sec:introduction-and-main-results}
In mathematical folklore there is a long tradition of \textit{hat problems}. Typically there are some logicians (or gnomes, for some odd reason) wearing a hat, and they are able to see at least some of the other logician's hats but not their own. Each player then guesses the colour of their own hat and they all win or lose collectively based on how good the guesses were. The consequences of losing for the logicians (or gnomes) are more or less drastic depending on who poses the problem. At first glance, it often seems impossible for the logicians to win consistently: how should the colour of the other hats help me guess my own? Nonetheless it is often possible to do so surprisingly well. We refer to the book \cite{HarTay13} for the history and many results on hat problems.

We will explore a specific type of hat problem where only a single logician has to guess correctly to win.
We start by defining the general format of the hat games and after that we will give some examples. Except for the last section, we will only assume relatively basic knowledge of ordinals and cardinals, but not more from the subjects of Mathematical Logic and Set Theory.
We hope that Section \ref{sec:hat-games-with-small-losing-threshold} will be accessible for the general mathematician, while Sections \ref{sec:hat-games-with-large-losing-threshold}, \ref{sec:proof-of-theorem-a} and \ref{sec:hat-games-on-partial-orders} will require some basic knowledge about ordinals and cardinals.

\begin{definition}
    Let $\lambda,\kappa\geq 1,\gamma\geq 2$ be cardinals. The $(\lambda,\kappa,\gamma)$-hat game is played as follows: there are $\lambda$ logicians, each of them wears a hat with an element from $\kappa$ written on it (which we will think of as one of $\kappa$-many colours).
    Each logician can look at a finite number of hats of other logicians, where the colours they see might influence at which hats and how many others they want to look at next. Then they each make a list of guesses, but the list has to have length strictly less than $\gamma$.
    The logicians win if there is at least one logician who listed the colour of their own hat. Of course, the logicians cannot see the colour of their own hat.\\
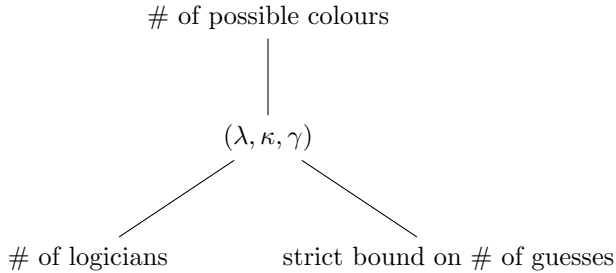
\begin{center}
\begin{tikzpicture}[scale=0.8]
\node (tuple) at (0, 0) {$(\lambda,\kappa, \gamma)$};
\node (lambda) at (-3, -2) {\# of logicians};
\node (kappa) at (0, 2) {\# of possible colours};
\node (gamma) at (3, -2) {strict bound on \# of guesses};

\draw[-] (kappa)--(tuple);
\draw[-] (lambda)--(tuple);
\draw[-] (gamma)--(tuple);
\end{tikzpicture}
\end{center}

    The logicians are allowed to have a strategy meeting before playing the game. We say that the $(\lambda,\kappa,\gamma)$-game is \textit{winning} for the logicians if there is a winning strategy for the logicians. Informally, a strategy is a set of instructions for each logician which, for a given valid sequence of hats, produces a list of ${<}\gamma$-many guesses in $\kappa$ depending on only finitely many other hats (and not their own hat!). It is winning if the logicians win against every possible allowed colouring of hats if every logician follows their set of instructions correctly. Formally, we equip $\kappa$ and $[\kappa]^{<\gamma}$ (the set of subsets of $\kappa$ of size $<\gamma$) with the discrete topology and say that a strategy is a sequence $\langle f_\alpha\mid\alpha<\lambda\rangle$ of continuous functions $f_\alpha\colon \kappa^{\lambda\setminus\{\alpha\}}\rightarrow[\kappa]^{<\gamma}$. If there is no winning strategy, the $(\lambda,\kappa,\gamma)$-game is \textit{losing}.
\end{definition}

We leave the exact logistics of realizing a $(\lambda,\kappa,\gamma)$-hat game to the imagination of the reader, though we propose one possible model: the logicians are all part of a big Zoom session and on Zoom, it is only possible to see finitely many other logicians at a time. Of course, they can scroll through the participants to reveal more, but it would take an infinite amount of time to see all other logicians (if $\lambda$ is infinite). We expect every logician to submit their guess after a finite amount of time, so if the logicians win they will do so after finitely much time has passed. This creates a strong incentive for the logicians to win: otherwise they might be stuck in a Zoom meeting forever.\medskip

The main difference from the type of hat problems we consider to the one which have been studied before is that the logicians have some agency on which hats to look at. In other problems, the hats whose colour a logician can take into account for their guess are predetermined by the rules of the game. It turns out that allowing the logicians to look at finitely many others is the sweet spot where this idea becomes interesting; if they are allowed to look at infinitely many others, it becomes too easy for the logicians to win, while a fixed finite amount makes it too hard.\\
Also, while we are not the first to consider this option, varying the amount of guesses the logicians are allowed to make seems to not have been explored a lot before.\medskip

Eventually, we will classify exactly for which $\lambda, \kappa, \gamma$, the $(\lambda,\kappa,\gamma)$-hat game is winning. Surprisingly, in many cases the logicians only start losing as soon as the number of colours is a \emph{MInA cardinal}. This is, in some sense, the remains of a certain large cardinal, an $\omega$-Erd\H{o}s cardinal. We now introduce the relevant notion of \emph{strongly MInA cardinals} \footnote{Coincidentally, Mina is the name of the first author's family dog. Pictures are available on request.}.
\begin{definition}
Let $\kappa$, $\gamma$ be infinite cardinals. 
Let $\mathcal F$ is a set of functions $f\colon \kappa^{n_f}\rightarrow [\kappa]^{<\gamma}$ for some finite $n_f$.
A set $A\subseteq\kappa$ is \textit{mutually $\mathcal F$-independent} if no $\alpha\in A$ can be covered by plugging in arguments from $A\setminus\{\alpha\}$ into functions in $\mathcal F$.
That is, if
\[\forall \alpha\in A\ \forall f\in\mathcal F\ \forall \alpha_1,\dots, \alpha_{n_f}\in A\setminus\{\alpha\}\ \alpha\notin f(\alpha_1,\dots,\alpha_{n_f}).\]
We say that $\kappa$ has the $\gamma$-\textbf{M}utually \textbf{In}dependent \textbf{A}ttribute ($\gamma$-MInA) or that $\kappa$ is a $\gamma$-MInA cardinal if for any countable such family $\mathcal F$, there is a countably infinite mutually $\mathcal F$-independent set $A\subseteq\kappa$.\\
A cardinal $\kappa$ is a \textit{strongly MInA}-cardinal if $\kappa$ has the $\gamma$-MInA for all $\gamma<\kappa$.
\end{definition}

\begin{remark}
While not quite immediate, the $\gamma$-MInA can be described equivalently in terms of an instance of the \textit{free subset property}: a cardinal $\kappa$ has the $\gamma$-MInA iff $\mathrm{Fr}_\gamma(\kappa,\omega)$ holds. Consequently, a cardinal $\kappa$ is strongly MInA iff $\mathrm{Fr}_\gamma(\kappa,\omega)$ holds for all $\gamma<\kappa$. We will discuss this later in Section \ref{sec:how-large-are-MInA-cardinals}. It will be more convenient for us to stick to the above terminology.
\end{remark}

Now we come to the first main Theorem of this paper.
\begin{restatable}{theorem}{maintheorem}\label{thm:all-games-classification}
    The $(\lambda,\kappa,\gamma)$-hat game is losing in exactly the following cases:
    \begin{enumerate}
        \item\label{case:gamecase1} $\lambda$ and $\gamma$ are finite and $\kappa \geq \lambda\cdot(\gamma-1)+1$.
        \item $\lambda$ is finite, $\gamma$ is infinite and $\kappa\geq \gamma^{+(\lambda-1)}$.
        \item\label{case:gamecase3} $\lambda$ is infinite, $\gamma$ is finite and $\kappa$ is infinite.
        \item\label{case:gamecase4} $\lambda$ and $\gamma$ are infinite and $\kappa$ is at least as large as the least strongly MInA cardinal above $\gamma$.
    \end{enumerate}
\end{restatable}

The most interesting case here is the case where $\lambda$ and $\gamma$ are infinite. 
Here are two interesting remarks that follow from the Theorem in this case:

\begin{itemize}
    \item First, the game does not get easier at all when more logicians are introduced, no matter how many, and
    \item secondly, increasing the number of allowed guesses only helps the logicians if we increase it at least to the next \textit{losing threshold}, i.e.~the least $\kappa$ so that the $(\lambda,\kappa,\gamma)$-game is losing for fixed $\lambda,\gamma$.
\end{itemize}
We remark that in cases \ref{case:gamecase1}-\ref{case:gamecase3}, the relevant losing threshold provably exists. In contrast, in case \ref{case:gamecase4}, the losing threshold, while it can be quite low (it can be $\aleph_\omega$ if measurable cardinals are consistent, see Fact \ref{fact:aleph_omega}), cannot be proved to exist in the standard axiom system $\mathrm{ZFC}$. Even if the losing threshold exists for a smaller $\gamma$, it may not exist for some larger $\gamma$. Indeed, the existence of such a losing threshold is quite a bit stronger than $\mathrm{ZFC}$ alone. We will discuss and prove this in Section \ref{sec:how-large-are-MInA-cardinals}. On the flipside, this means that there are set theoretical universes in which the logicians live happily: they can win any $(\lambda, \kappa, \gamma)$-game with $\lambda$ and $\gamma$ infinite.\medskip

 We will also consider these hat games in a general setting with a restriction on the vision of the logicians: logicians are forbidden from looking at other logicians whose index is smaller in some partial order.

 \begin{definition}
Suppose $\mathbb P$ is a partial order and $\lambda,\gamma$ are cardinals. In the $(\mathbb P,\kappa,\gamma)$-hat game there is one logician for each $p\in\mathbb P$. Each logicians' hat is coloured in one of $\kappa$-many colours. During the game, logician $p$ is allowed to look at the hats of finitely many other logicians, but only those logicians $q$ with $p<_{\mathbb P} q$. Finally, they submit a guess consisting of ${<}\gamma$-many colours. The logicians win collectively if one of the logicians included the colour of their hat in their guess. 
 \end{definition}
 
 This includes e.g.~hat games on $(\mathbb N, \leq)$: countably many logicians are standing in a line and can are allowed to look at finitely many logicians in front of them. 

 Among the variants of the $(\lambda,\kappa,\gamma)$-hat games which restrict the vision of the logicians further, the $(\mathbb P,\kappa,\gamma)$-hat games are exactly those in which (potential) vision is transitive: if logician $a$ is allowed to look at logician $b$ and logician $b$ is allowed to look at logician $c$ then logician $a$ is allowed to look at logician $c$ as well.

 Again, we classify exactly in which of the $(\mathbb P,\kappa,\gamma)$-hat games there is a winning strategy for the logicians.
 It turns out that most of the complexity is already captured by the unrestricted $(\lambda, \kappa, \gamma)$-hat games.
 Indeed, in all cases where $\mathbb P$ contains an infinite increasing sequence and $\gamma$ is infinite, the answer is the same.

 \begin{restatable}{theorem}{potheorem}\label{thm:all-partial-ordered-games-classification}
    Let $\mathbb P$ be a partial order. The $(\mathbb P,\kappa,\gamma)$-hat game is winning iff 
    \begin{enumerate}
        \item\label{case:pogamecase1} either $\kappa<\gamma$ or
        \item\label{case:pogamecase2} $\mathbb P$ contains an infinite increasing sequence, $\gamma$ is infinite and there is no strongly MInA cardinal in the interval $(\gamma,\kappa]$.
    \end{enumerate}
\end{restatable}

As a consequence of this classification, the existence of a winning strategy in the $(\mathbb P,\kappa,\gamma)$-hat game depends very little on the partial order $\mathbb P$ in question: we need only know whether or not $\mathbb P$ has an infinite increasing sequence.\medskip

 In the next Section we will consider the cases where $\gamma=2$ (meaning the logicians get exactly 1 guess) and $\lambda\leq\omega$ (i.e.~there are at most countably many logicians), as well as the case where $\lambda = 3$ and $\gamma=\omega$ (i.e.~there are $3$ logicians, each of which is allowed to make finitely many guesses).\medskip
 
 In Section \ref{sec:hat-games-with-large-losing-threshold}, we relate the $(\omega,\kappa,\gamma)$-hat games with infinite $\gamma$ to $\gamma$-MInA cardinals. We prove Theorem \ref{thm:all-games-classification} in the following Section \ref{sec:proof-of-theorem-a}.\medskip
 
 Finally, we solve the $(\mathbb P,\kappa,\gamma)$-hat games in Section \ref{sec:hat-games-on-partial-orders}. We finish with a discussion on the size of strongly MInA cardinals in Section \ref{sec:how-large-are-MInA-cardinals}, which also doubles as a short survey on the free subset property.

\section{Hat Games with Small Losing Threshold}\label{sec:hat-games-with-small-losing-threshold}

A great example of a hat puzzle, and the inspiration of this paper, is the following.
While many logical puzzles are folklore, this one can be attributed to Elliot Glazer\cite{Gla23}.
\begin{question}[{\bf The $(3,\mathbb R,\omega)$-Hat Game}]\label{que:three-people-real}

Three logicians are wearing a hat with a real number printed on it.
They see the hats that the other people are wearing.
They write down a finite list of real numbers.
If one of the logicians writes down the number on their hat, all logicians win.
Can they develop a strategy such that they will certainly win?
\end{question}
Some remarks are in order.
First of all, this puzzle takes place in some kind of ideal world, where every real number fits on a hat.
Moreover, we assume that the logicians have access to the full axiom of choice, i.e.~they can agree on a choice function for any set that they would like to use in their strategy.
Note that the logicians, of course, cannot see what the other logicians write down on their list.
We do not want to spoil the answer to Question \ref{que:three-people-real} now, but see Remark \ref{rem:three-people-real}.

We will now systematically study variants of this hat problem.
To do this we first consider a slightly simpler kind of problem: one where the logicians only get one guess, instead of finitely many.
\subsection{Finitely many logicians, one guess}

\begin{question}[\textbf{The $(n,A,2)$-Hat Game}]\label{que:finite-people-one-guess}

Let $n$ be a natural number, let $A$ be a non-empty set  of colours and consider the following puzzle: $n$ logicians wear a hat coloured in an element of $A$.
They can see the hats of all other people, but not their own.
Each logician writes down a colour from $A$ on a piece of paper.
If any logician writes down the colour of their own hat, they all win.
For which $n$ and $A$ do the logicians have a winning strategy?
\end{question}
\begin{answer}
There is a winning strategy if and only if $|A|\leq n$.
\end{answer}
\begin{proof}
Suppose that $|A|\leq n$.
We will describe a winning strategy for the logicians.
We might as well say that the colours are $A = \{0,1,\ldots,n-1\}$.
Number the logicians $0,1,\ldots,n-1$ and let the colours on their hat be $h_0,h_1,\ldots,h_{n-1}$.
Let $s$ be the image of $h_0+h_1 + \ldots+h_{n-1}$ in $\mathbb Z/n$.
Logician $i$ can see all hats except their own.
Summing all these values together modulo $n$ gives $s - h_i$.
Logician $i$ then writes down the colour $i - (s-h_i)$, modulo $n$.
To see that this strategy works, note that logician $s$ writes down $h_s$, and thus guesses correctly.

Now suppose that $|A|\geq n+1$, we will show that there is no winning strategy.
We can throw away all but $n+1$ elements of $A$, since that will make the game only easier.
Then we may also assume that $A = \{0,1,\ldots,n\}$.
Now suppose, for a contradiction, that there is some strategy $S$ for which the logicians always win.
colour each hat with a colour drawn uniformly at random from $A$, independently from the other hats.
The probability that any logician $i$ guesses correctly equals $\frac1{n+1}$: to see this, note that we can assign all other logicians their hat colours first.
Then we can already determine what logician $i$ will guess, and their hat colour is still uniformly random, so the probability that the hat colour equals the guess is $\frac1{n+1}$.

Now we use the additivity of the expectation value.
The expected value of a sum of random variables equals the sum of the expectation values, even if they are not independent.
The expected number of correct guesses by logician $i$ equals $\frac1{n+1}$ (this is just another fancy way of saying that the probability of guessing correctly is $\frac1{n+1}$).
So the expected number of correct guesses by all logicians equals $\frac n{n+1}$.
Since this is less then 1, there has to be at least one assignment of colours for which the strategy fails.
\end{proof}

\subsection{Infinitely many logicians, one guess}
Now we consider the same game, but there are countably infinitely many logicians. The following is a variation of a hat problem by Gabay-O'Connor. See the introduction of \cite{HarTay13} for the history of this problem.

\begin{question}\label{que:infinitely-many-logicians-not-continuous}
Let $A$ be a non-empty set of colours.
Infinitely many logicians, numbered $1,2,\ldots$ all wear a hat.
Each hat is coloured in an element of $A$.
The logicians see all other hats, and write down a guess for the colour of their own hat.
They win if at least one of them is correct.
For which $A$ can the logicians devise a winning strategy?
\end{question}
\begin{answer}
Surprisingly, the logicians can win the game for every set $A$.
Even more, they can devise their strategy in such a way, that whatever we write on the hats, all but finitely many logicians will be correct.
\end{answer}
\begin{proof}
The trick to solving this game is one that is used in many logic puzzles.
Consider the set $A^\omega$ of infinite sequences of elements of $A$.
On this set we introduce the equivalence relation $\sim$, given by: $(a_n)\sim (a'_n)$ if there is $N$ such that $a_n=a'_n$ for all $n\geq N$.
For each equivalence class $C \in A^\omega/\sim$ we choose an element $f(C) \in C$; this is part of the strategy that the logicians agree on beforehand.
Now, after the hat colours $h_n$ are distributed, all logicians can see the equivalence class $C$ of $(h_n)$ in $A^\omega/\sim$, since they see all hats but their own.
Let $(g_n) = f(C)$.
Then, since $(h_n)_{n\in\mathbb N}$ and $(g_n)_{n\in\mathbb N}$ are in the same equivalence class, there is some $N$ such that $h_n = g_n$ for $n\geq N$.
Now logician $n$ will guess that their own hat colour is $g_n$.
Then logician $n$ is correct if $n\geq N$, so all but finitely many logicians are correct.
\end{proof}

However, the situation changes if we only allow each logician to look at finitely many hats.
\begin{question}[\textbf{The $(\omega,A,2)$-Hat Game}]\label{que:infinitely-many-people-one-guess}

Let $A$ be a non-empty set of colours.
Infinitely many logicians, numbered $1,2,\ldots$ all wear a hat.
Each hat is coloured in an element of $A$.
Each logician can look at finitely many other hats, one by one, and then writes down a guess for the colour of their own hat.
The logicians win if at least one of them is correct.
For which $A$ can the logicians devise a winning strategy?
\end{question}
We note that the logicians do not have to decide in advance, which hats or how many hats they would like to see.
As an example, the first logician might look at the second hat, see the colour $n$, and then take a look at the $n$-th hat (if $n>1$).
The requirement that each logician looks at only finitely many hats is then equivalent to the ``guessing function" $A^\omega \to A$ being continuous, when we give $A$ the discrete topology and $A^\omega$ the product topology.

\begin{answer}
There is a winning strategy if and only if $A$ is finite.
\end{answer}
\begin{proof}
It is clear that the logicians can win if $A$ is finite; they can use the same strategy as discussed in Question \ref{que:finite-people-one-guess}.
So we just have to show that the logicians do not have a winning strategy if $A$ is infinite.
We may as well assume that the colours are $A = \omega$.
Suppose for a contradiction that the logicians have a winning strategy $S$.
We will again assign the hat colours randomly, but it is slightly more technical now, since we are working on an infinite space where not every set is measurable.

For each $n\in\omega$, let $p_n\colon \omega \to [0,1]$ be a probability measure satisfying $0 < p_n(m) \leq 2^{-n-2}$ for all $m\in\omega$.
We colour the hat of logician $n$ with a colour in $\omega$ according to the probability distribution $p_n$, independent from the hats of the other logicians.

For a finite subset $I\subseteq \omega$ and a function $f\colon I\to \omega$, we have a \emph{cylinder set} 
$$Z(f) = \{h\colon \omega\to \omega\mid h\text{ extends }f\}.$$
This is a measurable set and the probability that the random colour distribution is an element of $Z(f)$ equals $\prod_{i\in I}p_i(f(i))$.
Let $n$ be a logician and let $h$ be a hat function.
Let $I$ be the set of logicians that logician $n$ looks at to determine their guess.
Then for each hat function in the cylinder set $Z(h_{|I})$, logician $n$ will make the same guess.
Thus we know that the set of hat functions where logician $n$ guesses number $m$ can be written as a union of disjoint cylinder sets.
Since each cylinder set has a positive probability, this union is necessarily countable.
Thus in particular it is a measurable set, and the probability that logician $n$ guesses colour $m$ is well-defined.
Then the probability that logician $n$ guesses correctly is also well-defined.
Since $p_n(m)\leq 2^{-n-2}$ for all $m\in\omega$, the probability that the colour of the hat of logician $n$ equals the guess of logician $n$ can be at most $2^{-n-2}$.
Summing up over all $n$ gives a probability less than 1, hence the strategy cannot be always winning, and we have a contradiction.
\end{proof}
If we tried to apply this proof to Question \ref{que:infinitely-many-logicians-not-continuous}, it would not work because the set of hat functions where logician $n$ guesses colour $m$ is not measurable, and hence the probability that logician $n$ guesses correctly is not well-defined.

\subsection{Finitely many logicians, finitely many guesses}

Now we go back to the case where we have only finitely many logicians, but each logician can make finitely many guesses.
If there is only 1 logician, with a hat colour in $A$, the logician can only win if $A$ is finite, since they do not have any information.
Thus the first interesting case is with 2 logicians.
From now on we will assume that the set of possible hat colours is a cardinal, since only its cardinality is relevant anyway.

\begin{question}[\textbf{The $(2,\kappa,\omega)$-Hat Game}]\label{que:two-people-kappa}
Let $\kappa$ be a cardinal.
Three logicians are wearing a hat coloured in an element of $\kappa$.
They see the hats that the other logicians are wearing.
They write down a finite list of colours from $\kappa$.
If one of the logicians writes down the colour of their hat, they both win.
For which $\kappa$ can they develop a strategy such that they will certainly win?
\end{question}
\begin{answer}
    They can win if and only if $\kappa\leq \omega$.
\end{answer}
\begin{proof}
    Suppose $\kappa = \omega$.
    A winning strategy is as follows: each logician writes all colours up to and including the one that they see on the other logician's hat.
    Since the colour of one of the logicians has got to be smaller than the other, the corresponding logician below that hat will write down their colour.

    To see that $\kappa = \omega_1$ is not winning, we can make the game a little bit easier by promising to the first logician that the colour of their hat is actually an element of $\omega$.
    Given all possible hats of the first logician, the second logician will make only countably many guesses in total.
    Thus we can give the second logician a hat colour that they will never guess.
    After this we can determine the first logician's finitely many guesses and also give them a hat colour outside of their list.
    Then both will guess incorrectly.
\end{proof}

Now we can return to our original question.
\begin{question}[\textbf{The $(3,\kappa,\omega)$-Hat Game}]\label{que:three-people-kappa}
Let $\kappa$ be a cardinal.
Two logicians are wearing a hat coloured in an element of $\kappa$.
They see the hats that the other logician is wearing.
They write down a finite list of colours in $\kappa$.
If one of the logicians writes down the colour of their hat, they all win.
For which $\kappa$ can they develop a strategy such that they will certainly win?
\end{question}
\begin{answer}
    They can win if and only if $\kappa\leq \omega_1$.
\end{answer}
\begin{proof}
    We consider the case $\kappa = \omega_1$.
    All logicians agree beforehand on an embedding $i_\alpha\colon\alpha\curvearrowright\omega$ for each countable ordinal $\alpha$.
    Now the strategy of each logician is as follows: look at the largest of the other two logicians' hat colours, call it $\alpha$.
    Then the colour on the other visible hat $\beta$ is an element of $\alpha + 1$.
    Then write down the list $i_{\alpha + 1}^{-1}([0,i_{\alpha+1}(\beta)])$.
    This strategy works because there is one logician with the largest hat colour, and after applying the embedding of its successor ordinal into $\omega$, one of the other logicians has the smallest hat colour and that one guesses correctly.

    Now we show that $\kappa = \omega_2$ is not winning.
    We can make the game easier by promising the first logician that their hat is coloured in an element of $\omega$ and promising the second logician that their hat colour is an element of $\omega_1$.
    All possible guesses of the third logician together have only a cardinality of at most $\omega_1$.
    Thus we can give the third logician a hat colour that they will never guess.
    After this, all possible guesses of the second logician will make a countable set.
    So we can still give the second logician a hat colour that they will never guess, while keeping our promise.
    After that the finitely many guesses of the first logician are determined so we can also give them a hat colour in $\omega$ that they will not guess.
\end{proof}

\begin{remark}\label{rem:three-people-real}
    This shows that the answer to the original Question \ref{que:three-people-real} is independent of $\mathrm{ZFC}$: indeed, its affirmation is equivalent to the Continuum Hypothesis, that states that the cardinality of the reals equals $\omega_1$.
\end{remark}

The case for $n$ logicians is very similar: indeed, they can win the $(n,\kappa,\omega)$-hat game exactly if $\kappa\leq \omega_{n-2}$ (see Theorem \ref{thm:all-games-classification}).

\section{Hat Games with Large Losing Threshold}\label{sec:hat-games-with-large-losing-threshold}

In this Section we finally make relate hat games to MInA cardinals. We show that the $(\omega,\kappa,\gamma)$-hat game is losing exactly if $\kappa$ is a $\gamma$-MInA cardinal.\\
It will be convenient for us to replace finitary functions $f\colon\kappa^{n_f}\rightarrow[\kappa]^{<\gamma}$ by continuous functions $f\colon\kappa^\omega\rightarrow[\kappa]^{<\gamma}$ in the definition of $\gamma$-MInA. Here, the topology on $\kappa^\omega$ is the product topology and $\kappa$ as well as $[\kappa]^{<\gamma}$ are construed as discrete spaces.

\begin{lemma}\label{lem:mina-with-cont-fcts}
A cardinal $\kappa$ has the $\gamma$-MInA if and only if for every countable set $\mathcal F$ of continuous functions $f\colon\kappa^\omega\rightarrow[\kappa]^{<\gamma}$ there is an infinite set $A$ so that $\alpha\notin \bigcup f[(A\setminus\{\alpha\})^\omega]$ for all $\alpha\in A$.
We will also say that such sets $A$ are mutually $\mathcal F$-independent.
\end{lemma}

\begin{proof}
Clearly, this property implies the $\gamma$-MInA since any function $f\colon\kappa^n\rightarrow[\kappa]^{<\gamma}$ can be considered as a continuous function $\kappa^\omega\rightarrow[\kappa]^{<\gamma}$ by forgetting about all but the first $n$ arguments.

Now suppose that $\kappa$ has the $\gamma$-MInA and we will show that this stronger property also holds. Suppose $\mathcal F=\{f_n\mid n<\omega\}$ is a countable set of continuous functions $f_n\colon\kappa^\omega\rightarrow[\kappa]^{<\gamma}$. For $n, k<\omega$, define 
$$f_{n}^k\colon \kappa^k\rightarrow[\kappa]^{<\gamma}$$
by 
\[f_n^k(\alpha_1,\dots,\alpha_k)=\begin{cases}
    X & \text{if } f_n(s)=X \text{ for any }s\in\kappa^\omega \text{ beginning with }\alpha_1,\dots,\alpha_n\\
    \emptyset & \text{ else.}
\end{cases}\]
Then $\mathcal F'=\{f_n^k\mid n, k<\omega\}$ is a countable set of finitary functions on $\kappa$, so there is a countable mutually $\mathcal F'$-independent set $A\subseteq\kappa$. If $\alpha\in A$, $n<\omega$ and $s\in (A\setminus\{\alpha\})^\omega$ then 
$f_n(s)=f_n^k(s(0), s(1), \dots, s(k-1))$ for any sufficiently large $k$ by continuity of $f_n$. It follows that $\alpha\notin f_n(s)$, so $A$ is mutually $\mathcal F$-independent.
\end{proof}

\subsection{Countably many logicians, finitely many guesses}

\begin{lemma}\label{lem:countably-many-strategies}
    Consider an $(\omega,\kappa,\gamma)$-game with $\gamma\geq\omega$.
    Suppose there is a sequence of continuous strategies $\tau^N$, such that for all hat functions $h\colon\omega\to\kappa$, there is $N<\omega$ and $i<\omega$ such that $h(i)\in \tau^N_i(h)$.
    Then this game is winnable.
\end{lemma}
\begin{proof}
    We describe a winning strategy.
    Divide the logicians in two groups, label them by elements of $\omega\times\{0,1\}$.
    For $b\in\{0,1\}$ let $g_b\colon \omega \to \omega\times\{0,1\}$ sending $i$ to $(i,b)$.
    Let $N_b$ be the minimal number such that there is $i<N_b$ and $N<N_b$ with $h(i,b) \in \tau^N_i(h\circ g_b)$.
    
    Any logician in group $b$ can determine $N_{1-b}$ as follows: for an integer $N$, look at the hats $N\times\{1-b\}$.
    For each pair $(i,N')\in N^2$, see what logician $i$ would do for the strategy $\tau^{N'}$ and the hat function $h\circ g_{1-b}$.
    Since we only try to apply the strategy to finitely many logicians, we only have to look at finitely many hats.
    We find out the list of guesses that logician $i$ would write down, and we can check whether this list contains $h(i,1-b)$.
    If so, then $N_{1-b} \leq N$.
    Going through all integers $N$ one by one, we will eventually find $N_{1-b}$ in finite time.

    Now the logician in group $b$ tries each strategy $\tau^N$ for $N<N_{1-b}$.
    They write down the union of all guesses they make.
    This is a finite union so the list of guesses is still less than $\gamma$.
    Now if $N_0\leq N_1$, one of the logicians in group 0 will guess correctly, and if $N_1\leq N_0$, one of the logicians in group 1 will guess correctly.  
    
\end{proof}

\begin{question}[The $(\omega,\kappa,\omega)$-Game]
Let $\kappa$ be a cardinal. Countably infinite many logicians are wearing a hat coloured in an element of $\kappa$. Each logician can look at finitely many other hats and then writes down a finite list of colours in $\kappa$. If one of the logicians writes down their hat colour, all the logicians win. For which $\kappa$ do the logicians have a winning strategy?
\end{question}

\begin{answer}\label{answer:omega-logicians-finitely-many-guesses}
They can win if and only if $\kappa$ is not an $\omega$-MInA cardinal.
\end{answer}

\begin{proof}
First suppose the logicians have a winning strategy.
For $n<\omega$, let $f_n\colon\kappa^{\omega-\{n\}}\rightarrow[\kappa]^{<\omega}$ describe the strategy of logician $n$, so that the sequence $\langle f_n\mid n<\omega\rangle$ is winning for the logicians.
To be more precise, if $h\in\kappa^{\omega}$ so that $h(n)$ is the hat colour of logician $n$ then the list that logician $n$ writes is $f_n(h\upharpoonright \omega\setminus\{n\})$.
Then $\mathcal F=\{f_n\mid n<\omega\}$ is a countable collection of continuous functions and we will show that there is no countably infinite mutually $\mathcal F$-independent set, so that $\kappa$ is not $\omega$-MInA by Lemma \ref{lem:mina-with-cont-fcts}.
So let $A\subseteq\kappa$ be an infinite subset and let $h\in A^\omega$ be a sequence of hat colours drawing from $A$ without repetition.
As the logicians win, there is some $n$ so that the colour $h(n)$ of logician $n$'s hat appears in their list $f_n(h\upharpoonright\omega\setminus\{n\})$ of guesses.
As $h$ does not repeat a colour, $h(n)$ is different from all colours in $h\upharpoonright\omega\setminus\{n\}$ and consequently, $A$ is not mutually $\mathcal F$-independent.\medskip

For the other direction, assume that $\kappa$ is not a $\omega$-MInA cardinal.
This means that there is a countable collection $\mathcal F=\{f_n\mid n<\omega\}$ of functions $f\colon\kappa^{n_f}\rightarrow[\kappa]^{<\omega}$ so that there is no countably infinite mutually $\mathcal F$-independent set.
We will go on to define countably many strategies by describing how logician $n$ comes up with their list of guesses in each of them.  
 
For a given $N<\omega$, we will now describe a strategy $\tau^N$ for the logicians. Suppose $h\in\kappa^\omega$ is a colouring of the logicians' hats. 

\noindent Intuitively, logician $n$ assumes that for some $j\leq N$, 
\begin{equation*}
h(n)\in f_j(\alpha_1,\dots, \alpha_m)\text{ for some hats }\alpha_1,\dots,\alpha_m\in h[\{0,\dots, N\}\setminus\{n\}].
\end{equation*}
 Logician $n$ can use this idea to come up with a finite list of guesses which contains their true hat colour if this assumption is correct. This list is 
$$\bigcup_{j\leq N}\bigcup_{n\neq k_1\leq N,\dots, n\neq k_l\leq N} f_j(h(k_1),\dots, h(k_l)).$$
Of course, logician $n$ need only look at the hats of the first $N+1$ other logicians to calculate this list. Now suppose that $h\in\kappa^\omega$ is a colouring of hats which uses infinitely many colours. The set $\mathrm{ran}(h)$ cannot be mutually $\mathcal F$-independent in this case. It follows that we can find $n, j, l<\omega$ so that $h(n)\in f_j(\alpha_1,\dots, \alpha_m)$ for some hat colours $\alpha_1,\dots,\alpha_m\in h[\{0,\dots, l\}\setminus\{n\}]$ so that logician $n$ guesses correctly when following $\tau^{\max\{j, l\}}$.

We describe one more strategy $\tau_\ast$ which wins against any colouring using only finitely many colours: under $\tau_\ast$, logician $n$ simply lists all the colours on hats of logicians with index ${<}n$.

In any case, one of the countably many strategies $\tau_\ast, (\tau_N)_{N<\omega}$ wins against every colouring, so that the $(\omega,\kappa,\omega)$-hat game is winning by Lemma \ref{lem:countably-many-strategies}.
\end{proof}

The same argument works for any $(\omega,\kappa,\gamma)$-game with $\gamma\geq\omega$ and we get:
\begin{question}[The $(\omega,\kappa,\gamma)$-Game]\label{que:omega-kappa-infinite-game}
Let $\kappa$ and $\gamma\geq\omega$ be cardinals.
Countably infinite many logicians each get a hat coloured in a colour from $\kappa$.
Each logician is allowed to look at finitely many other logicians' hats and has to come up with a list of ${<}\gamma$-many colours.
The logicians win collectively if and only if one of the logicians includes their hat colour in their list of guesses.
For which $\kappa,\gamma$ can the logicians always win?
\end{question}

\begin{answer}
The logicians have a winning strategy if and only if $\kappa$ does not have the $\gamma$-MInA.
\end{answer}

\section{Proof of Theorem \ref{thm:all-games-classification}}\label{sec:proof-of-theorem-a}
In this Section we will proof our main Theorem.
We have already done a lot of steps in proving the main Theorem.
The main challenges left are showing what happens when we increase the number of logicians, and increasing the number of guesses they are allowed to make.

\subsection{Uncountably many logicians}
We will now discuss what happens when the amount of logicians is increased and $\gamma$ is infinite.
A plausible guess is that the losing threshold is related to a version of the $\gamma$-MInA in which the size of the mutually independent set is increased to the amount of logicians.
However, this is not the case and we will prove that a larger number of logicians does not help them win at all.
First we need to be clear what it means to have an uncountable number of logicians.
Recall that in the case $\lambda=\omega$, we look at \emph{continuous} strategies, where the strategy of logician $n$ can be described by a continuous function $\kappa^{\omega\setminus n} \to [\kappa]^{<\gamma}$.
It is fairly easy to see that this can be combinatorically interpreted as the logician looking at the hats of the other logicians one by one, and having to decide after looking at finitely many hats.
It turns out that a similar interpretation holds in the uncountable case.

\begin{proposition}\label{prop:continuous-functions-many-logicians}
Let $\lambda$ be a cardinal, $X$ a discrete space and $f\colon\kappa^\lambda\to X$ a function.
The following are equivalent:
\begin{enumerate}
    \item The function $f$ is continuous.
    \item The function can be described by a procedure that takes $h\in\kappa^\lambda$, evaluates one $h(\alpha)$ at a time and returns $f(h)$ after evaluating finitely many values of $h$.
\end{enumerate}
\end{proposition}

\begin{proof}
The function $f$ is continuous exactly if for all $h\in\kappa^\lambda$ there is a finite subset $F_h\subseteq \lambda$ such that all $h'\in \kappa^\lambda$ with $h'\upharpoonright F_h=h\upharpoonright F_h$ satisfy $f(h')=f(h)$.
Then the direction $(ii)\implies(i)$ is obvious, since we can take $F_h$ to be the finitely many values of $f$ that the procedure has evaluated on.
Now suppose that $f$ is continuous, choose an appropriate $F_h$ for each $h\in\kappa^\lambda$.
We describe a successful procedure that describes $f$:\\
\underline{Step 0:} Let $h_0\in \kappa^\alpha$ be arbitrary and let $x_0 = f(h_0)$.
Let $S_0 = F_{h_0}$ and evaluate $h$ on all elements of $S_0$.\\
\underline{Step $n+1$:} Consider all possible $h'\in\kappa^\alpha$ that agree with $h$ on $S_n$.
If for all these $h'$ we have $f(h') = x_n$, return the output $x_n$.
Otherwise pick some $h_{n+1}$ that agrees with $h$ on $S_n$ and with $f(h_{n+1}) = x_{n+1} \neq x_n$.
Let $S_{n+1} = S_n \cup F_{h_{n+1}}$, evaluate $h$ at all the points of $S_{n+1}$ and go to the next step.

It is clear that if this procedure finishes at some point, then we do so with the correct output $f(h)$.
We must show that this process terminates after finitely many steps.
So suppose that this did not happen with input $h$.
Let $S=\bigcup_{n<\omega}S_n$, the set of indices we have looked at during the above procedure.
If $n$ is sufficiently large then we have 
\[S\cap F_h = S_n \cap F_h.\]
For such $n$, we know that $h_{n+1}$ agrees with $h$ on $S_n$ and that $S_{n+1}\cap F_h = S_n\cap F_h$.
Thus there is a function $h'$ with $h'\upharpoonright S_{n+1} = h_{n+1}$ and $h'\upharpoonright F_h = h$.
Then we have $f(h') = f(h_{n+1}) = x_{n+1}$ and $f(h') = f(h)$, so $f(h) = x_{n+1}$.
But the same argument shows that $f(h) = x_{n+2}$, giving a contradiction.
\end{proof}

Thus now we can say that, even in the case of uncountable many logicians, we consider continuous strategies, and this is equivalent to saying that each logician can look at finitely many other logicians before guessing their own hat.
Now, before proving that adding more logicians does not help them, we need the following Lemma.
\begin{lemma}\label{lem:partly-unwinnable-implies-unwinnable}
Consider an $(\lambda,\kappa,\gamma)$-hat game.
Suppose that for every strategy of the logicians, there is a non-empty subset $Y\subseteq \lambda$ and a function $h\colon Y\to\kappa$, such that if the logicians in $Y$ have their hats coloured by $h$, they only look at other logicians in $Y$ and guess incorrectly.
Then the $(\lambda,\kappa,\gamma)$-hat game is losing.    
\end{lemma}
\begin{proof}
    Note that the condition immediately holds for $(L,\kappa,\gamma)$-hat games for all non-empty subsets $L \subseteq \lambda$ as well.
    Fix a strategy for the logicians.
    We construct a hat function recursively for which the logicians will lose the game.
    For ordinals $\alpha$, define $Y_\alpha\subseteq \lambda$ and $h_\alpha\colon Y_\alpha\to\kappa$ as follows: if $\alpha$ is a limit cardinal, put $Y_\alpha=\bigcup_{\alpha'<\alpha}Y_{\alpha'}$ and $h_\alpha = \bigcup_{\alpha'<\alpha}h_{\alpha'}$.
    Consider a successor ordinal $\alpha+1$ and suppose $Y_\alpha \neq\lambda$.
    Fixing the hat colours of the logicians in $Y_\alpha$, the logicians in $\lambda\setminus Y_\alpha$ have a strategy for the $(\lambda\setminus Y_\alpha, \kappa,\gamma)$-hat game.
    By the condition, there is a non-empty subset $Y'\subseteq \lambda\setminus Y_\alpha$ and a function $h'\colon Y'\to\kappa$, such that given these hats, the logicians in $Y'$ only look at the hats in $Y_\alpha\cup Y'$ and guess incorrectly.
    Then put $Y_{\alpha+1} = Y_\alpha\cup Y'$ and $h_{\alpha+1} = h_\alpha\cup h'$.

    Now there must be some $\alpha$ with $Y_\alpha = \lambda$, and with the hat function $h = h_\alpha$, all logicians guess incorrectly.
\end{proof}

Strategies for the $(\lambda,\kappa,\gamma)$-hat game are simply sequences $\langle f_\alpha\mid\alpha<\lambda\rangle$ of continuous functions $f_\alpha\colon\kappa^{\lambda\setminus\{\alpha\}}\rightarrow[\kappa]^{<\gamma}$.
In the proof of Proposition \ref{prop:amount-of-logicians-irrelevant}, we want to make use of how the game would play out ``in practice": a given logician starts by looking at some other logician, at which point their hat colour is revealed.
Now, taking this information into account, they choose to look at another logician's hat, etc.~until they decide to stop and guess after finitely many steps. 
Indeed, this procedure defines the same class of strategies by Proposition \ref{prop:continuous-functions-many-logicians}.

\begin{proposition}\label{prop:amount-of-logicians-irrelevant}
    Suppose $\lambda$ and $\gamma$ are infinite cardinals and the $(\lambda,\kappa,\gamma)$-hat game is winning.
    Then the $(\omega,\kappa,\gamma)$-hat game is winning.
\end{proposition}
\begin{proof}

 Suppose for a contradiction that the $(\omega,\kappa,\gamma)$-hat game is losing, so we know that $\kappa$ has the $\gamma$-MInA.
Fix a winning strategy for the $(\lambda, \kappa,\gamma)$-hat game and let us assume from now on that all logicians follow this strategy.
For $h\in \kappa^n$, let $f_\alpha(h)$ be the guess that logician $\alpha$ submits when the first hat that they look at has colour $h(0)$, the second hat they look at has colour $h(1)$, etc.~until they decide to make the guess; if they do not make a guess after seeing $n$ hats, we leave $f_\alpha(h)$ undefined.
For this definition, it does not matter at which logicians $\alpha$ looks at, only the colours that they see.
For $\alpha<\lambda$ and $h_0\in \kappa^n$, let $\mathrm{next}_\alpha( h_0)$ denote, if defined, the index of the logician that logician $\alpha$ wants to look at next after seeing the hat colours $h_0(0),\dots, h_0(n-1)$.
Define 
\[\mathrm{next}_\alpha(h_0,\dots, h_{n+1})=\mathrm{next}_{\mathrm{next}_\alpha(h_0,\dots,  h_n)}(h_{n+1})\]
by recursion on $n$ for $h_0,\dots, h_{n+1}\in \kappa^{<\omega}$ (whenever the right hand side is defined). 
Now, for $\alpha<\lambda$ and $n_0,\dots n_m\in\omega$, let 
\[f_\alpha^{n_0,\dots, n_m, n}\colon \kappa^{n_0}\times\dots\times\kappa^{n_m}\times\kappa^n\rightarrow [\kappa]^{<\gamma}\]
be given by 
\[(h_0,\dots, h_m, h)\mapsto f_{\mathrm{next}_\alpha(h_0,\dots, h_m)}(h)\]
when the right-hand side is defined, and the empty set otherwise.
We identify the domain of $f_\alpha^{n_0,\dots, n_m, n}$ with $\kappa^{n_0+\dots+n_m+n}$.
Set 
\[\mathcal F=\bigcup_{m<\omega}\{f_0^{n_0,\dots, n_m, n}\mid n_0,\dots, n_m, n<\omega\}\]
and let $A$ be a countably infinite mutually $\mathcal F$-independent set.
Let 
\[Y=\{0\}\cup\bigcup_{m<\omega}\{\mathrm{next}_0(h_0,\dots, h_{m})\mid h_0,\dots, h_m\in A^{<\omega}\}\]
and note that $Y$ is countable so that there is an injection $h\colon Y\rightarrow A$.
Now let $h^+$ be any colouring extending $h$.
By definition of $Y$, if any logician $\alpha\in Y$ sees only colours in $A$ then they will only ever look at other logicians in $Y$.
As the hats of logicians in $Y$ have colours in $A$, this shows that logicians in $Y$ only look at other logicians in $Y$ when confronted with $h^+$.
We will now prove that all logicians in $Y$ guess incorrectly.
It is easier to see that $0$ guesses incorrectly, so consider a logician
\[\alpha=\mathrm{next}_0(h_0,\dots, h_m)\in Y\setminus\{0\}\]
for some $h_0,\dots, h_m\in A^{<\omega}$.
We may assume that $m$ is as small as possible and the length of $h_m$ is as small as possible.
It follows that $h(\alpha)$ does not appear in the $h_0,\dots, h_m$. 
We know that logician $\alpha$ guesses 
\[f_0^{n_0,\dots, n_m, n}(h_0,\dots, h_m, h')\]
where $n_i=\mathrm{length}(h_i)$ for $i<m$ and $n$ is the amount of logicians that $\alpha$ looks at, and for $k<n$, $h'(k)$ is the colour of the $k$-th hat that logician $\alpha$ chooses to look at.
In particular we have $\mathrm{ran}(h')\subseteq A\setminus\{h(\alpha)\}$.
It follows that logician $\alpha$ does not guess correctly by our choice of $A$.\\
 By Lemma \ref{lem:partly-unwinnable-implies-unwinnable}, this implies that the $(\lambda,\kappa,\gamma)$-hat game is losing, a contradiction. So the $(\omega,\kappa,\gamma)$-hat game is winning.
\end{proof}

\begin{remark}
We do not know of a more direct proof of the above proposition, i.e.~a proof which somehow transforms a winning strategy of a $(\lambda,\kappa,\gamma)$-game into one for the $(\omega,\kappa,\gamma)$-game.
\end{remark}

\subsection{More guesses}

\iffalse Now we show that given a strategy for a game with hats in $\kappa$ and $\delta$ many guesses, and a strategy for hats in $\delta$ and $<\gamma$ many guesses, we can combine them into a strategy for a game with hats in $\kappa$ and $<\gamma$ many guesses.
\fi

\begin{lemma}\label{lem:combining-strategies-of-different-cardinalities}
    Suppose the games $(\omega,\kappa,\delta^+)$ and $(\omega,\delta,\gamma)$ are winnable, with $\delta,\gamma\geq\omega$.
    Then the game $(\omega,\kappa,\gamma)$ is winnable as well.
\end{lemma}
While it is possible to prove this Lemma now, it will be much easier to do this after we have introduced the FInA in Section \ref{sec:hat-games-on-partial-orders}.
Therefore, we postpone the proof for now.
The proof is on page \pageref{proof:combining-strategies-of-different-cardinalities}.
We use this Lemma in the remainder of this chapter, but we do not use it in the next Section before proving it; thus the reader may convince themselves that our arguments are non-circular.

We say $\kappa$ is \emph{strongly losing} if for all $\gamma<\kappa$ the hat game $(\omega,\kappa,\gamma)$ is losing.
The following Proposition follows directly from Lemma \ref{lem:combining-strategies-of-different-cardinalities}.

\begin{proposition}\label{prop:least-strongly-losing}
Let $\gamma\geq\omega$ be a cardinal.
The least $\kappa$ for which the hat game $(\omega,\kappa,\gamma)$ is losing equals the least strongly losing cardinal above $\gamma$.
\end{proposition}
\begin{proof}
    Clearly $\kappa\geq \gamma$.
    In fact, $\kappa>\gamma$, otherwise the first two logicians can simply look at each other's hats, and write down all hat colours up to and including the other player's hat's.
    By definition, any strongly losing cardinal larger than $\gamma$ also loses the hat game $(\omega,\kappa,\gamma)$.
    So let $\kappa$ be the least cardinal losing the game $(\omega,\kappa,\gamma)$; we will prove that $\kappa$ is strongly losing and then we will be done.
    Suppose that $\kappa$ is not strongly losing.
    Then there is $\delta<\kappa$ such that $(\omega,\kappa,\delta)$ is winning.
    By minimality of $\kappa$, the hat game $(\omega,\delta,\gamma)$ is also winning.
    By Lemma \ref{lem:combining-strategies-of-different-cardinalities}, the hat game $(\omega,\kappa,\gamma)$ is winning, contradiction.
\end{proof}
We can express this in terms of the MInA.
Recall that a cardinal $\kappa$ is strongly MInA just in case that $\kappa$ is a $\gamma$-MInA cardinal for all $\gamma<\kappa$.
As an immediate consequence of the answer to Question \ref{que:omega-kappa-infinite-game} we get:

\begin{corollary}\label{cor:stongly-losing-is-strongly-mina}
A cardinal $\kappa$ is strongly losing if and only if $\kappa$ is strongly MInA.
\end{corollary}

This is the fact that we have hinted at earlier: a cardinal has the $\omega$-MInA iff it has the $\gamma$-MInA for any/all $\gamma$ below the least $\omega$-MInA cardinal.\medskip

Now we are ready to prove Theorem \ref{thm:all-games-classification}, which we repeat for convenience.
\maintheorem*
\begin{proof}
\begin{enumerate}
    \item The case with $\lambda$ and $\gamma$ finite is similar to Question \ref{que:finite-people-one-guess}.
    \item Suppose $\lambda$ is finite and $\gamma$ is infinite.
    We provide a winning strategy for $\kappa< \gamma^{+(\lambda -1 )}$ by induction on $\lambda$.
    The case $\lambda = 1$ is trivial as the single logician can guess all colours in $\kappa$.
    Suppose that for each ordinal $\mu < \gamma^{+(\lambda-1)}$ we have a strategy $\tau_\mu$ for $\lambda$ players and hat colours in $\mu$.
    We describe a strategy for $\lambda+1$ players and hats in $\kappa < \gamma^{+\lambda}$.
    Each logician $i$ looks at the largest hat colour they can see, say logician $j$ with hat $h(j)$.
    Then the other logicians that logician $i$ can see have hat colours in $h(j)+1$.
    Since $h(j)+1 < \kappa < \gamma^{+\lambda}$ we have $h(j)+1 < \gamma^{+(\lambda-1)}$, logician $i$ can apply strategy $\tau_{h(j)+1}$ on the logicians except logician $j$.
    This gives logician $i$ their list of guesses.
    One of the logicians has the hat with the largest colour, and using this strategy, one of the other logicians has to guess correctly.

    Now we show that for $\kappa \geq \gamma^{+(\lambda-1)}$ there is no winning strategy.
    We again use induction on $\lambda$.
    The case $\lambda = 1$ is clear, since the single logician does not have any information and can not guess all possible colours of their hat.
    So suppose there is no winning strategy for some $\lambda$ and consider a strategy for the $(\lambda+1, \kappa, \gamma)$-hat game, with $\kappa \geq \gamma^{+\lambda}$.
    We can promise the first $\lambda$ players that their hat's colour will be an element of $\gamma^{+(\lambda-1)}$.
    Then there are only $\gamma^{+(\lambda-1)}$ possible hat colourings over the first $\lambda$ players.
    In each case the last player guesses less than $\gamma$ hat colours, so in total, they guess at most $\gamma^{+(\lambda-1)}$ colours.
    Thus we can give the last player a hat colour that they will certainly not guess.
    Now after fixing this hat colour, the first $\lambda$ players are playing a $(\lambda,\gamma^{+(\lambda-1)},\gamma)$-hat game, for which we know there is no winning strategy.
    Thus we can also give them hat colours for which they will all guess incorrectly, so the $(\lambda+1,\kappa,\gamma)$-hat game is losing.    
    \item Now consider the case with $\lambda$ infinite and $\gamma$ finite.
    If $\lambda$ is countable, the proof is similar to that of Question \ref{que:infinitely-many-people-one-guess}.
    In the general case we make use of Lemma \ref{lem:partly-unwinnable-implies-unwinnable}.
    Consider any strategy for the $(\lambda,\omega,\gamma)$-hat game.
    Now we construct sets $X_n = \{x_0,x_1,\ldots,x_{m_n}\}$ recursively.
    Let $X_0=\{x_0\}$ where $x_0$ is any logician.
    If $X_n$ has been defined, consider all hat colourings where the hat colour of logician $x_i$ is smaller than $2^{i+2}\gamma$ for all $i\leq m_n$.
    For each logician $x_i$ with $i\leq m_n$ and each such hat colouring, add the first logician that $x_i$ looks at that is outside of $X_n$ to $X_{n+1}$.
    
    Let $Y$ be the union of all $X_n$.
    Then, if for all $i$ the hat of logician $x_i$ is smaller than $2^{i+2}\gamma$, all logicians of $Y$ look only at other logicians of $Y$.
    Now give all logicians in $Y$ a random hat colour, where the colour of the hat of logician $x_i$ is drawn uniformly from $\{0,\ldots,2^{i+2}\gamma-1\}$ and the hat colours are mutually independent.
    Since each logician always looks at only finitely many other hats, the probability that logician $x_i$ guesses correctly is well-defined.
    Since the guess of logician $x_i$ does not depend on their own hat, and their own hat's colour is uniformly chosen from $2^{i+2}\gamma$ elements, the probability that they are correct is at most $\frac1{2^{i+2}}$.
    Then the probability that one of them guesses correctly is at most $\frac12$.
    So there is some colouring of hats on the logicians in $Y$ such that the logicians in $Y$ only look at each other and all of them guess incorrectly.
    By Lemma \ref{lem:partly-unwinnable-implies-unwinnable}, the $(\lambda,\omega,\gamma)$-hat game is losing.
    \item Finally, suppose that $\lambda$ and $\gamma$ are infinite.
    By Proposition \ref{prop:amount-of-logicians-irrelevant}, the $(\lambda,\kappa,\gamma)$-hat game is losing if and only if the $(\omega,\kappa,\gamma)$-hat game is losing.
    By the answer to Question \ref{que:omega-kappa-infinite-game}, this happens if and only if $\kappa$ has the $\gamma$-MInA.
    By Proposition \ref{prop:least-strongly-losing} and Corollary \ref{cor:stongly-losing-is-strongly-mina}, this holds if and only if $\kappa$ is at least the least strongly MInA cardinal above $\gamma$.
\end{enumerate}
\end{proof}

\section{Hat Games with Restrictions}\label{sec:hat-games-on-partial-orders}

In the literature about hat problems, there are often restrictions places on which other logicians a logician is allowed to look at.
For example, there can be infinitely many logicians standing in a line, where each logician may only see the logicians in front of them.
The situation can be quite complicated, as is shown for example in \cite[Section 4.4]{HarTay13}.
There we consider the situation where there are infinitely many logicians indexed by $\omega$, where odd-numbered logicians can see all even-numbered logicians in front of them, and even-numbered logicians can see all odd-numbered logicians in front of them.
The answer to the hat puzzle, where the logicians can look at infinitely many hats, is complicated and independent of ZFC.
For us it is much simpler.
\begin{question}
    Consider $\omega$ many logicians, each wearing hat with a colour in $\kappa$.
    Suppose that each logician can look at finitely many hats, and then write down a finite list of guesses for their own hat.
    Suppose furthermore that logicians can only look at logicians whose index is higher and of a different parity than their own.
    The logicians win if at least one of them guesses correctly.
    For which $\kappa$ do they have a winning strategy?
\end{question}
\begin{answer}
    The logicians can win if $\kappa \leq \omega$.
\end{answer}
\begin{proof}
    Suppose $\kappa=\omega$.
    Logician $n$ looks at the next logician, and guesses all hat colours up to and including the hat colour of logician $n+1$.
    Since there is no infinite decreasing sequence in $\omega$, one of the logicians has to guess correctly.

    Suppose $\kappa=\omega_1$ and suppose for a contradiction that the logicians have a winning strategy.
    We can promise the logicians that all even logicians will get a colour in $\omega$.
    Now the guesses of the odd logicians depend only on the hat colours of the even logicians, and since there are $\omega$ many finite sequences of elements in $\omega$, we know that each odd logician can only make countably many guesses.
    So we can give each odd logician a hat that they will never guess.
    After that, the guesses of the logicians are fixed, so we can give them a hat colour in $\omega$ that they will never guess.
    So all logicians guess wrong, a contradiction.
\end{proof}

In general, analysing the hat problem for each directed graph is complicated, even in the finite case.
If we consider any finite directed graph $G$, we can consider the hat game where the logicians are vertices of $G$, and they can see another logician if there is an arrow from them to that logician.
In \cite{AloBenShaTam20}, the finite case is considered, where each logician can make 1 guess.
Then the game is winnable up to some finite number of colours.
If we allow the logicians to make finitely many guesses, the game will be winnable whenever the amount of colours is at most $\omega_n$, for some $n$.
It would be interesting, but out of scope for this paper, to see how $n$ depends on the directed graph.

For the remainder of this Section we will consider the \emph{transitive} case, where if logician $a$ can see logician $b$ and logician $b$ can see logician $c$, then logician $a$ can see logician $c$.
Since no logician can see themselves this implies in particular that there are no cycles.
Thus the visibility restrictions can be viewed as a partial order.
We will now analyse for which partial orders $\mathbb P$ and cardinals $\kappa,\gamma$ the $(\mathbb P,\kappa,\gamma)$-hat game is winning.
Recall that in this game, there is one logician for every $p\in\mathbb P$. 
In the game, logician $p$ is allowed to look at finitely many other hats, but only at those which belong to a logician $q$ with $p<_{\mathbb P}q$.
As before, the hat colours are drawn from $\kappa$ and each logician is allowed to guess ${<}\gamma$-many colours.
The win condition is as before: at least one logician needs to guess correctly.

In the standard $(\lambda,\kappa,\gamma)$-hat game, the existence of a winning strategy depended only very little on $\lambda$.
A similar situation arises here, too.
In fact, there is really only one partial order we have to look at explicitly to solve the whole problem.

\subsection{Well-ordered hat games and FInA cardinals}

In the original formulation by Gabbay-O'Connor of the hat problem in Question \ref{que:infinitely-many-logicians-not-continuous}, the $\omega$-many logicians are standing in a line ordered by their index and can only see the logicians in front of them.
This corresponds to the restriction of playing on the partial order $(\omega,\leq)$.
We will now discuss the $((\omega,\leq),\kappa,\gamma)$-hat games in the case that $\gamma$ is infinite.
We will call this game the \textit{well-ordered} $(\omega,\kappa,\gamma)$-hat game.

\begin{question}[The well-ordered $(\omega,\kappa,\gamma)$-Game]\label{que:wo-omega-kappa-infinite-game}
Suppose $\gamma\geq\omega$. Countably many logicians are standing in a line of length $\omega$, the logicians cannot see logicians behind them. Every logician wears a hat coloured in an element of $\kappa$ and is allowed to look at finitely many other logicians (in front of them). Every logician writes down finitely many guesses and if one logician guesses the colour of their hat, all logicians win. For which $\kappa$ can the logicians win?
\end{question}

Surprisingly, the additional restriction on their vision does not matter at all to the logicians. 

\begin{answer}
The logicians win if and only if they win the unrestricted $(\omega,\kappa,\gamma)$-hat game, that is if and only if $\kappa$ does not have the $\gamma$-MInA.
\end{answer}

Our proof will proceed as follows: first, we describe the losing threshold of this game in terms of a property which is similar to the $\gamma$-MInA and then prove that it is actually equivalent the $\gamma$-MInA. Once again, we do not know of a direct argument which translates a winning strategy of the $(\omega,\kappa,\gamma)$-game into a winning strategy for the well-ordered $(\omega,\kappa,\gamma)$-game.\\
The relevant property is the $\gamma$-FInA.

\begin{definition}
Suppose $\kappa$, $\gamma$ are cardinals. If $\mathcal F$ is a collection of functions $f\colon \kappa^{n_f}\rightarrow[\kappa]^{<\gamma}$ then a set $A\subseteq\kappa$ is \textit{forwards $\mathcal F$-independent} if no $\alpha\in A$ can be generated by larger elements of $A$ via functions in $\mathcal F$, i.e.
\[\forall\alpha\in A\forall f\in\mathcal F\forall \beta_1,\dots,\beta_{n_f}\in A\setminus (\alpha+1)\ \alpha\notin f(\beta_1,\dots,\beta_{n_f}).\]
We say that $\kappa$ has the $\gamma$-\textit{Forwards Independent Attribute} or $\gamma$-\textit{FInA} if and only if for any countable set $\mathcal F$ as above there is a countably infinite forwards $\mathcal F$-independent set $A\subseteq\kappa$.
\end{definition}

For well-ordered $(\omega,\kappa,\gamma)$-hat games we have a version of Lemma \ref{lem:countably-many-strategies}.
In fact it is easier to prove in this case.
\begin{lemma}\label{lem:countably-many-strategies-well-ordered}
    Consider a well-ordered $(\omega,\kappa,\gamma)$-game with $\gamma\geq\omega$.
    Suppose there is a sequence of continuous, forward-looking strategies $\tau^N$, such that for all hat functions $h\colon\omega\to\kappa$, there is $N<\omega$ and $i<\omega$ such that $h(i)\in \tau^N_i(h)$.
    Then this game is winnable.
\end{lemma}
\begin{proof}
    We can assume that there are infinitely many logicians that guess correctly for at least one of the $\tau^N$, e.g.~by splitting the $\omega$ many logicians into infinitely many infinite groups and applying the strategies to all of the groups individually.
    We describe one winning strategy that combines all the $\tau^N$.
    Each logician $i$ determines the minimal $N$ such that a logician between $i+1$ and $i+N$ guesses correctly for a strategy $\tau^k$ with $k\leq N$.
    They can do it because they know that there is such an $N$, and for each $N$ they can check if it is the case while only looking at finitely many hats of logicians with a higher index.
    Then logician $i$ guesses the union of $\tau^k_i(h)$ over all $k\leq N$.

    There is a logician $i$ for which the smallest $N$ with $h(i) \in \tau^N_i(h)$ is minimal, and that logician guesses correctly when using the strategy above.
\end{proof}

\begin{lemma}
For $\gamma\geq\omega$, the logicians win the well-ordered $(\omega,\kappa,\gamma)$-hat game if and only if $\kappa$ has the $\gamma$-FInA.
\end{lemma}

\begin{proof}
The argument is basically the same as the solution to Question \ref{que:omega-kappa-infinite-game}. When constructing strategies, we roughly have to adjust by looking at the next $N$ logicians if we had previously looked at the first $N$ logicians. That these adjustments work out corresponds exactly to the weakening from mutually independent to forwards independent.
\iffalse Proving a version of Lemma \ref{lem:countably-many-strategies} is even easier in this context: we do not have to split into groups. Instead, logician $n$ can find some $N$ so that the logicians in front win when playing according to the $N$-th strategy. Then they should guess according to all the first $N$-many strategies themselves.
\fi
We leave the exact details to the reader.
\end{proof}

The punchline is that the $\gamma$-FInA is the same as the $\gamma$-MInA.

\begin{lemma}\label{lem:FInA-MInA-equivalence}
Let $\gamma>1$ be a cardinal.
A cardinal $\kappa$ has the $\gamma$-FInA if and only if it has the $\gamma$-MInA.
\end{lemma}
At the heart of the proof is a modification of an argument due to Koepke \cite{Koe84}. We also mention that the result is not obvious: there are certainly forwards independent sets which are not mutually independent.
To eliminate the case in which $\gamma$ is singular, we need another small result first.

\begin{proposition}\label{prop:FInA-step-up}
Suppose $\kappa,\gamma$ are infinite cardinals and $\kappa$ has the $\gamma$-FInA. Then $\kappa$ has the $\gamma^+$-FInA. 
\end{proposition}

\begin{proof}
We prove this using hat games.
The statement is equivalent to this: suppose the well-ordered hat game is winnable with $\kappa$ many colours and $\gamma$ many guesses; then it is also winnable with $\kappa$ many colours and $<\gamma$ many guesses.
Let $\tau$ be a strategy for the logicians with $\gamma$ many guesses.
Note that necessarily infinitely many logicians guess correctly for any possible hat colouring.
We describe a winning strategy $\tau'$ where each logician makes $<\gamma$ guesses.
For each logician $n$ and hat function $h$, we view $\tau_n(h)$ as a function $\tau_n(h):\gamma\to\kappa$.
Let $\delta_n(h)$ be the least element of $\gamma$ such that $h(n) = \tau_n(h)(\delta_n(h))$, if it exists.
Now each logician $n$ can determine the least logician $m>n$ that guesses correctly when following $\tau$.
Then $\delta_m(h)$ exists and can be determined by logician $n$.
Now logician $n$ guesses the restriction of $\tau_n(h)$ to $\delta_m(h)+1$.

Logician $n$ guesses correctly if they guessed correctly for the strategy $\tau$ and $\delta_n(h) \leq \delta_m(h)$.
Since there is no infinite decreasing sequence of ordinals, there must be a logician that guesses correctly.

\end{proof}

\begin{proof}[Proof of Lemma \ref{lem:FInA-MInA-equivalence}]
It is clear that $\gamma$-MInA implies $\gamma$-FInA, so assume that $\kappa$ is $\gamma$-FInA.
It is straightforward to see that at least $\omega$-FInA holds, so assume $\gamma\geq\omega$.
Further, we may have that $\gamma$ is regular as otherwise we can replace $\gamma$ by the regular cardinal $\gamma^+$ by Proposition \ref{prop:FInA-step-up}.
Let $\mathcal F$ be a countable collection of functions $f\colon \kappa^{n_f}\rightarrow [\kappa]^{<\gamma}$. Since $\gamma$ is regular and by possibly enlarging $\mathcal F$, we may assume without loss of generality that 
\begin{enumerate}[label=($\mathcal F$.\roman*)]
    \item $\mathcal F$ is closed under composition in the sense that if $f, f_1,\dots, f_{k}\in \mathcal F$ with $k=n_f$ then the function $g\colon \kappa^{\sum_{i=1}^k n_{f_i}}\rightarrow [\kappa]^{<\gamma}$ is in $\mathcal F$ where 
    $$g(\alpha_1^1,\dots, \alpha^1_{n_{f_1}},\dots, \alpha_1^{k},\dots, \alpha_{n_{f_{k}}}^{k})=\bigcup_{\delta_1\in f_1(\alpha_1^1,\dots, \alpha^1_{n_{f_1}})}\dots\bigcup_{\delta_{k}\in f_{k}(\alpha_1^{k},\dots, \alpha_{n_{f_{k}}}^{k})} f(\delta_1,\dots, \delta_{k}),$$
    \item $\mathcal F$ is closed under permutations of arguments, that is if $f\in\mathcal F$ and $\sigma\in S_{n_f}$ is a permutation then $(\alpha_1,\dots,\alpha_{n_f})\mapsto f(\alpha_{\sigma(1)},\dots,\alpha_{\sigma(n_f)})$ is in $\mathcal F$ and
    \item\label{item:F-closed-under-minimal-witnesses} for any $f\in \mathcal F$ and $m<n_f$ there is $r\in\mathcal F$ so that
    $$\beta_1,\dots,\beta_m\in r(\alpha, \alpha_1,\dots, \alpha_n)$$
    where $\beta_m\geq\dots\geq\beta_1$ are lexicographically least\footnote{This means $\beta_m$ is minimized first, then $\beta_{m-1}$, and so on.} with 
    $$\alpha\in f(\beta_1,\dots, \beta_m,\alpha_1,\dots,\alpha_n)$$
    if such ordinals exist.
\end{enumerate}

By assumption, there is $A\subseteq\kappa$ an infinite forwards $\mathcal F$-independent set. Next, we construct by recursion on $i<\omega$ a strictly increasing sequence $\langle\alpha_i\mid i<\omega\rangle$ of elements in $\kappa$ as well as finite subsets $\langle B_i\mid i<\omega\rangle$ of $A$. We will make sure that 
\begin{equation*}\tag*{$(\ast)_{i}$}\label{eq:alpha-i-property}
\{\alpha_{0},\dots, \alpha_{i}\}\cup A\setminus\left(\bigcup_{k=0}^i B_{k}\right)\text{is forwards $\mathcal F$-independent}
\end{equation*}
for all $i<\omega$.
We will also make sure that the least element of $A\setminus\left(\bigcup_{k=0}^i B_k\right)$ is larger than $\alpha_i$.
Suppose $\alpha_j, B_j$ are defined for all $j<i$. Then we set $\alpha_i$ to be the least $\beta$ so that $\exists B\in[A]^{<\omega}$ with the property that
$$\{\alpha_{0},\dots, \alpha_{i-1}\},\ \{\beta\},\  A\setminus\left(\bigcup_{k=0}^{i-1} B_{k}\cup B\right)$$
\begin{itemize}
    \item are listed in order, i.e.~any member of a set listed earlier is strictly smaller than every member of a set listed later and
    \item their union is forwards $\mathcal F$-independent.
\end{itemize}
We then set $B_i$ to be some such witness $B$ and note that \ref{eq:alpha-i-property} is satisfied. This construction does not break down: the minimal left-over point $\beta=\min A\setminus \bigcup_{j<i} B_j$ satisfies the requirement with $B=\{\beta\}$.\\
It follows from \ref{eq:alpha-i-property} for all $i<\omega$ that $\{\alpha_i\mid i<\omega\}$ is forwards $\mathcal F$-independent and we aim to show that it is even fully mutually $\mathcal F$-independent. Suppose towards a contradiction that this is not the case, so that there are 
$$\alpha_{k_0}<\dots<\alpha_{k_a}<\alpha_i<\alpha_{l_0}<\dots<\alpha_{l_b}$$
and some $f\in\mathcal F$ with 
$$\alpha_i\in f(\alpha_{k_0},\dots, \alpha_{k_a},\alpha_{l_0},\dots,\alpha_{l_b}).$$
By \ref{item:F-closed-under-minimal-witnesses}, there are 
$$\beta_0\leq\dots\leq\beta_a\leq\alpha_{k_a}$$
and some $r\in\mathcal F$ with 
\begin{enumerate}[label=($\vec\beta$.\roman*)]
    \item\label{item:alphai-generated-by-betas} $\alpha_i\in f(\beta_0,\dots,\beta_a,\alpha_{l_0},\dots,\alpha_{l_b})$ and
    \item\label{item:betas-generated-by-alphas} $\beta_0,\dots,\beta_a\in r(\alpha_i,\alpha_{l_0},\dots,\alpha_{l_b})$.
\end{enumerate}
By \ref{item:betas-generated-by-alphas}, no $\beta_k$ can be among the $\{\alpha_j\mid j<\omega\}$. We set 
$$B=\bigcup_{k=0}^{l_b}B_k.$$
Observe that not all $\beta_k$, $k\leq a$, can be generated by functions in $\mathcal F$ using parameters in $A\setminus B$: otherwise, by \ref{item:alphai-generated-by-betas} and since $\mathcal F$ is closed under composition, $\alpha_i$ can be generated from a function in $\mathcal F$ by parameters in 
$$\{\alpha_{l_0},\dots,\alpha_{l_b}\}\cup A\setminus B,$$
which contradicts $(\ast)_{l_b}$. So let $\beta$ be one of the $\beta_k$ not generated in this way. Find the least $j\leq i$ with $\beta<\alpha_j$. We will reach a contradiction by showing that an ordinal ${\leq}\beta$ should have been picked instead of $\alpha_j$ at step $j$ of the above construction. 
It is clear that 
$$\{\alpha_{0}, \dots, \alpha_{j-1}\},\  \{\beta\},\  A\setminus\left(\bigcup_{k=0}^{j-1} B_{k}\cup B\right)=A\setminus B$$
is listed in order. Let $U$ be the union of the above sets. For $\delta\in U$, say that $U$ is \textit{forwards independent at} $\delta$ if 
$$\forall f\in\mathcal F\ \forall \delta_1,\dots,\delta_m\in U\setminus (\delta+1)\ \delta\notin f(\delta_1,\dots,\delta_m).$$
We surely have that $U$ is forwards independent at any $\delta\in U$ above $\beta$ since $A$ is forwards $\mathcal F$-independent. By our choice of $\beta$, $U$ is also forwards independent at $\beta$. Finally, $U$ is forwards independent at all $\alpha_{0},\dots,\alpha_{j-1}$: if some $\alpha_{m}$, $m<j$, could be generated via a function in $\mathcal F$ from parameters in 
$$\{\alpha_{m+1},\dots,\alpha_{j-1}\}\cup\{\beta\}\cup A\setminus B$$
then $\alpha_{m}$ could be generated via a function in $\mathcal F$ from parameters in 
$$\{\alpha_{m+1},\dots,\alpha_{j-1}\}\cup\{\alpha_i,\alpha_{l_1},\dots,\alpha_{l_b}\}\cup A\setminus B$$
as $\beta$ can be generated from $\{\alpha_i,\alpha_{l_1},\dots,\alpha_{l_b}\}$ from a function in $\mathcal F$ and $\mathcal F$ is closed under composition. Once again, this is in contradiction to $(\ast)_{l_b}$. We have shown that $\{\alpha_i\mid i<\omega\}$ is mutually $\mathcal F$-independent.
\end{proof}

It follows that our answer to Question \ref{que:wo-omega-kappa-infinite-game} is correct.

Now we consider the well-ordered analogue of Lemma \ref{lem:combining-strategies-of-different-cardinalities}.
It turns out that this one is much easier to prove.
\begin{lemma}\label{lem:combining-strategies-FINA-case}
    Suppose the well-ordered games $(\omega,\kappa,\delta^+)$ and $(\omega,\delta,\gamma)$ are winnable, with $\delta,\gamma\geq\omega$.
    Then the well-ordered $(\omega,\kappa,\gamma)$-game is winnable as well.
\end{lemma}
\begin{proof}
    Let $\tau$ be a strategy for the well-ordered $(\omega,\kappa,\delta^+)$-hat game.
    For $h\in\kappa^\omega$ and $n\in\omega$ we consider $\tau_n(h)$ as a function $\delta\to\kappa$.
    Let $\rho$ be a strategy for the well-ordered $(\omega,\delta,\gamma)$-hat game.
    For each finite $E\subseteq\omega$ we describe a strategy $\nu^E$ for the well-ordered $(\omega,\kappa,\gamma)$-hat game, such that for all hat functions $h\in\kappa^\omega$, there is an $E$ and an $n$ with $h(n) \in \nu^E(n)$.
    By Lemma \ref{lem:countably-many-strategies-well-ordered}, this shows that the well-ordered $(\omega,\kappa,\gamma)$-hat game is winnable.

    We only need to describe $\nu^E_n$ for $n\in E$.
    Let $e\colon |E|\to E$ be the increasing enumeration of $E$.
    Logician $n$ determines $\tau_m(h)$ for all $m\in E$ with $m>n$.
    For such $m$, logician $n$ can determine if $h(m)$ lies in the image of $\tau_m(h)$.
    If this is not the case for some such $m$, logician $n$ just guesses the empty set for this strategy.
    If this is the case for every such $m$, let $\alpha_m$ be such that $\tau_m(h)(\alpha_m) = h(m)$.
    Now let $h'\colon |E|\setminus (e^{-1}(n)+1) \colon \delta$ be given by $h'(i)  = \alpha_{e(i)}$.
    Logician $n$ now imagines playing as logician $e^{-1}(n)$ in the well-ordered $(\omega,\delta,\gamma)$ game with the hat function $h'$.
    Logician $n$ can try to determine $\rho_{e^{-1}(n)}(h')$.
    If the strategy $\rho$ requires them to look at a hat function of a too large integer, for which $h'$ is not defined, then they just guess the empty set again.
    Otherwise, they also determine $\tau_n(h)$, and they finally guess the set
    \[\nu^E_n(h) = \tau_n(h)(\rho_{e^{-1}(n)}(h')).\]

    To show that this works, fix a hat function $h$ and let $W\subseteq \omega$ be the set of logicians that win for the strategy $\tau$.
    Let $w\colon \omega\to W$ be the increasing enumeration of $W$.
    For $n\in\omega$ there is $h'(n)$ with $h(w(n)) = \tau_{w(n)}(h)(h'(n))$.
    There is an $n$ with $h'(n) \in \rho_n(h')$.
    Furthermore there is an $N$ such that in the well-ordered $(\omega,\delta,\gamma)$-hat game, logician $n$ only looks at logicians smaller than $N$ (and bigger than $n$).
    Let  $E = W \cap w(N)$.
    Then for the strategy $\nu^E$, logician $w(n)$ guesses $\tau_{w(n)}(h)(\rho_{e^{-1}(w(n))}(h')) = \tau_{w(n)}(h)(\rho_n(h'))$, which contains $\tau_{w(n)}(h)(h'(n)) = h(w(n))$.
    So logician $w(n)$ guesses correctly for strategy $\nu^E$.
\end{proof}

\begin{proof}[Proof of Lemma \ref{lem:combining-strategies-of-different-cardinalities}]\label{proof:combining-strategies-of-different-cardinalities}
This now follows directly from Lemma \ref{lem:combining-strategies-FINA-case} and the answer to Question \ref{que:wo-omega-kappa-infinite-game}.
Now we can really view Theorem A as proved.
It was more easy to prove the Lemma in the well-ordered case, because each logician can determine exactly what each logician in front of them will guess; by contrast, in the non-restricted game, it can be impossible to determine what any other logician will guess, because they can look at your own hat.
\end{proof}

\subsection{Proof of Theorem \ref{thm:all-partial-ordered-games-classification}}
We prove our second main result, classifying exactly which $(\mathbb P,\kappa,\gamma)$-hat games are winning.
\potheorem*

\begin{proof}
The $(\mathbb P, \kappa, \gamma)$-hat game is clearly winning if $\kappa<\gamma$ as the logicians can guess all colours in this case. So from now on assume $\gamma\leq\kappa$.
\begin{description}
    \item[Case 1:] $\mathbb P$ has no infinite increasing sequence.  We must show that the $(\mathbb P,\kappa,\gamma)$-hat game is losing, so suppose for a contradiction that $\tau$ is some strategy for the logicians.
    The case assumption implies that every non-empty subset of $\mathbb P$ has a maximal element. We define a colouring $h\colon \mathbb P\rightarrow\kappa$ by transfinite induction. Start by picking a maximal element $p_0\in \mathbb P$ and note that $p_0$ does not have any access to information during the game, so we can define $h(p_0)$ to be a colour not guessed by $p_0$. Now if $p_\beta$ is defined for all $\beta<\alpha$, stop if $\mathbb P\setminus\{p_\beta\mid\beta<\alpha\}=\emptyset$, otherwise pick a maximal element $p_\alpha$ from this set. During the game, $p_\alpha$ is only ever allowed to look at logicians $p_\beta$ with $\beta<\alpha$. As the colour of their hat is already chosen, the guess of $p_\alpha$ according to $\tau$ is determined, so we may let $h(p_\alpha)$ be a colour not guessed.\\ 
    Eventually, we have defined the colouring $h$ on all of $\mathbb P$ and $\tau$ loses against $h$.    
    \item [Case 2] $\gamma$ is finite. Once again, we show that the logicians lose. We may assume that $\kappa=\gamma$, as decreasing $\kappa$ makes it only easier on the logicians. Let us assume the logicians follow a fixed strategy $\tau$.
\begin{claim*}
For any finite $I\subseteq \mathbb P$, there is a colouring $h\colon\mathbb P\rightarrow\kappa$ so that all logicians in $I$ guess incorrectly.
\end{claim*}
\begin{proof}
We proceed by induction on the size of $I$. The claim holds trivially if $\vert I\vert=0$, so let us show the inductive step. Let $I\subseteq \mathbb P$ be of size $n+1$. As $I$ is a finite partial suborder, it must have a minimal point $p\in I$. Now, there is a colouring $h\colon\mathbb P\rightarrow \kappa$ so that all logicians in $I\setminus\{p\}$ guess incorrectly. Next, change the colour of logician $p$'s hat to a colour not guessed by $p$. As none of the logicians in $I\setminus\{p\}$ could look at the hat of logician $p$, their guess remains unchanged, so they guess incorrectly and so does $p$ now.
\end{proof}
For each finite $I\in [\mathbb P]^{<\omega}$, choose a colouring $h_I\colon\mathbb P\rightarrow\kappa$ so that all logicians in $I$ guess incorrectly. We have just described a net $h_\bullet=(h_I)_{I\in[\mathbb P]^{<\omega}}$ (we order finite subsets of $\mathbb P$ by inclusion). Since $\kappa$ is finite, $\kappa^{\mathbb P}$ is compact, so the net $h_\bullet$ has a cluster point $h_\ast$ which $\tau$ must lose against: the guess that logician $p$ submits against $h_\ast$ depends only on the finite set $I$ of logicians that $p$ looked at. The set 
$$O=\{g\in\kappa^{\mathbb P}\mid \forall i\in I\ g(i)=h_\ast(i)\}$$
is open in $\kappa^{\mathbb P}$ with $h_\ast\in O$. Hence there must be some $\{p\}\subseteq J\in[\mathbb P]^{<\omega}$ with $h_J\in O$. Logician $p$ guesses the same colours against $h_J$ and $h_\ast$. Since they guesses incorrectly against $h_J$ by choice of $h_J$, logician $p$ does not guess their own colour against $h_\ast$ either.

\item[Case 3:] $\mathbb P$ contains an infinite increasing sequence and $\gamma$ is infinite. We show that the logicians have a winning strategy iff there is a strongly MInA cardinal in the interval $(\gamma,\kappa]$. We can now ``sandwich" the $(\mathbb P,\kappa,\gamma)$-game between two games we have already analysed: 
\begin{align*}
   &\text{the well-ordered }(\omega,\kappa,\gamma)\text{-game is winning} \\
   \Rightarrow & \text{the }(\mathbb P,\kappa,\gamma)\text{-game is winning}\\
   \Rightarrow & \text{the }(\vert\mathbb P\vert,\kappa,\gamma)\text{-game is winning.}
\end{align*}
Here, the first implication follows from the existence of an infinite increasing sequence in $\mathbb P$; the logicians can employ a winning strategy from the well-ordered $(\omega,\kappa,\gamma)$-game on that sequence. The second implication holds as the $(\vert\mathbb P\vert,\kappa,\gamma)$-game is the same as the $(\mathbb P,\kappa,\gamma)$-game without the restriction on vision. The punchline is that since $\gamma$ is infinite, we know from Theorem \ref{thm:all-games-classification} and our answer to Question \ref{que:wo-omega-kappa-infinite-game} that 
\begin{align*}
   &\text{the well-ordered }(\omega,\kappa,\gamma)\text{-game is winning} \\
   \Leftrightarrow & \text{there is no strongly MInA cardinal in the interval }(\gamma,\kappa]\\
   \Leftrightarrow & \text{the }(\vert\mathbb P\vert,\kappa,\gamma)\text{-game is winning,}
\end{align*}
so this also characterises when the $(\mathbb P,\kappa,\gamma)$-game is winning.
\end{description}
\end{proof}

This is as far as our methods will take us. We make extensive use of transitivity of vision in many arguments. This transitivity does not quite hold for the standard $(\lambda,\kappa,\gamma)$-game as logician $a$ may look at logician $b$, who looks back at $a$, yet logician $a$ cannot see their own hat. However, this is the only way in which this can fail and one can work around this small issue, e.g.~in Lemma \ref{lem:countably-many-strategies} with the ``splitting into two groups" trick. One can consider the $(\lambda,\kappa,\gamma)$-hat game more generally with potential visibility governed by a directed graph, but this is considerably more complicated. Already in the case of finitely many logicians and only one guess each a full classification seems completely out of reach. See e.g.~\cite{AloBenShaTam20}, \cite{Bra22} and \cite{HeIdoPrz22} for recent results on this topic.

\section{How large are MInA cardinals?}\label{sec:how-large-are-MInA-cardinals}

So far, we have not said much at all about the size of strongly MInA cardinals. We hope that the reader does not feel we simply proved one opaque combinatorial property to be equivalent to another. The notion of a strongly MInA cardinal is much more tangible from a set theoretic perspective than that of a strongly losing cardinal, as we will see shortly. In fact, there is a good body of literature on the size of the least cardinal $\kappa$ with the free subset property $\mathrm{Fr}_\omega(\kappa,\omega)$. To access this, we prove the equivalence we have mentioned before.

\begin{definition}
    Let $\kappa,\gamma,\lambda$ be cardinals. The free subset property $\mathrm{Fr}_{\gamma}(\kappa,\lambda)$ holds iff for any collection $\mathcal F$ of functions $f\colon \kappa^{n_f}\rightarrow\kappa$ of size $\gamma$, there is a $\mathcal F$\textit{-free} set of size $\lambda$. A $\mathcal F$-free set is a subset $A$ of $\kappa$ which satisfies 
    $$\forall \alpha\in A\forall f\in\mathcal F\forall \alpha_1,\dots,\alpha_{n_f}\in A\setminus\{\alpha\}\ \alpha\neq f(\alpha_1,\dots,\alpha_{n_f}).$$
\end{definition}

\begin{lemma}
    Let $\kappa,\gamma$ be infinite cardinals. $\kappa$ has the $\gamma$-MInA iff $\mathrm{Fr}_\gamma(\kappa,\omega)$ holds.
\end{lemma}

\begin{proof}
    Suppose $\kappa$ has the $\gamma$-MInA.
    As an immediate consequence of either Theorem \ref{thm:all-games-classification} or the combination of Proposition \ref{prop:FInA-step-up} and Lemma \ref{lem:FInA-MInA-equivalence}, we know that $\kappa$ has the $\gamma^+$-MInA as well.
    Let $\mathcal F$ be a collection of $\gamma$-many functions $f\colon \kappa^{n_f}\rightarrow\kappa$. Now define
    $$F_n\colon\kappa^n\rightarrow[\kappa]^{\leq\gamma}$$
    by $F(\alpha_1,\dots,\alpha_n)=\{f(\alpha_1,\dots,\alpha_n)\mid f\in\mathcal F\wedge n_f=n\}$.
    Next, any mutually $\{F_n\mid n<\omega\}$-independent set is $\mathcal F$-free, so $\mathrm{Fr}_{\gamma}(\kappa,\omega)$ holds.
    
    Now assume that $\mathrm{Fr}_\gamma(\kappa,\omega)$ holds and $\mathcal F$ is a countable set of functions $f\colon \kappa^{n_f}\rightarrow [\kappa]^{<\gamma}$. For $\nu<\gamma$ and $f\in\mathcal F$, define 
    $$f^\nu\colon \kappa^{n_f}\rightarrow\kappa$$
    by 
    $$f^\nu(\alpha_1,\dots,\alpha_n)=\begin{cases}
    \beta & \text{if $\beta$ is the $\nu$-th element of $f(\alpha_1,\dots,\alpha_n)$}\\
    0 & \text{if $f(\alpha_1,\dots,\alpha_n)$ has ordertype ${<\nu}$.}
    \end{cases}$$
    Then it is clear that any $\{f^\nu\mid f\in\mathcal F,\nu<\gamma\}$-free set is mutually $\mathcal F$-independent.
\end{proof}

Unfortunately, we now have to assume a little more knowledge on Mathematical Logic and Set Theory. 

So, what can be said about size of the least strongly MInA cardinal? First of all, we should note that $\mathrm{ZFC}$ does not prove strongly MInA cardinals exist, so it is possible that all $(\omega,\kappa,\omega)$-games are winning for the logicians. In this sense, strongly MInA cardinals are large. The reasons for possible non existence is that their existence has some large cardinal strength. It is a theorem of Baumgartner and independently Devlin-Paris \cite{DevPar73} that the exact consistency strength of the existence of an $\omega$-MInA cardinal  (equivalently a cardinal $\kappa$ with $\mathrm{Fr}_\omega(\kappa,\omega)$) in terms of large cardinals is the existence of the $\omega$-Erd\H{o}s cardinal. We will go on and characterise the strongly MInA cardinals in G\"odels constructible universe $L$, the smallest \text{inner model} of $\mathrm{ZFC}$. It follows from absoluteness of well-foundedness that any strongly MInA cardinal is strongly MInA in $L$: for example if $\tau\in L$ is a strategy for the $(\omega,\kappa,\gamma)$-game which is losing in the true universe $V$ then the tree $T\in L$ of attempts to find a hat colouring which $\tau$ loses to is ill-founded in $V$ and hence must be ill-founded in $L$. Any infinite branch through $T$ corresponds to a full colouring that $\tau$ loses against. This combined with the upcoming characterisation of strongly MInA cardinals in $L$ proves this equiconsistency.

\begin{fact}[Baumgartner, Devlin-Paris \cite{DevPar73}]
The theories 
\begin{enumerate}
    \item $\mathrm{ZFC}+``\text{there is a $\omega$-MInA cardinal}"$,
    \item $\mathrm{ZFC}+``\text{the }\omega\text{-Erd\H{o}s cardinal exists}"$
\end{enumerate}
are equiconsistent.
\end{fact}

On the other hand, MInA cardinal can be quite small in terms of the $\aleph$-hierarchy and continuum function.

\begin{fact}[Devlin \cite{Dev73}]
If $\mathrm{ZFC}+``\text{there is a strongly MInA cardinal}"$ is consistent then so is the theory $\mathrm{ZFC}+``2^{\omega}\text{ is strongly MInA}"$.
\end{fact}

We remark once again that Devlin phrased his result in terms of the free subset property.

Since $\gamma$-MInA implies $\gamma^+$-MInA, strongly MInA cardinals are uncountable limit cardinals. So the least strongly MInA cardinal is at least $\aleph_\omega$. Shelah has proven that this cannot be improved.

\begin{fact}[Shelah \cite{She80}]
If $\mathrm{ZFC}$ is consistent with the existence of infinitely many measurable cardinals then so is ``$\aleph_\omega$ is strongly MInA".
\end{fact}

Koepke has refined this to an equiconsistency.

\begin{fact}\label{fact:aleph_omega}[Koepke \cite{Koe84}]
The theories 
\begin{enumerate}
    \item $\mathrm{ZFC}+``\aleph_\omega \text{ is a strongly MInA cardinal}"$,
    \item $\mathrm{ZFC}+``\text{there is a measurable cardinal}"$
\end{enumerate}
are equiconsistent.
\end{fact}

In this sense, strongly MInA cardinals can be quite small: they neither need to be regular nor strong limit cardinals. Neither does their existence imply the existence of inaccessible cardinals. Nonetheless, the least strongly MInA cardinal in $L$ and other canonical inner models is inaccessible and much more, which we will prove soon.

We further remark that measurable cardinals and even Ramsey cardinals are all strongly MInA. They are even limits of strongly MInA cardinals, so they do not function as losing thresholds.

We should also note that some instance of the axiom of choice is necessary to prove that strongly MInA cardinals are limit cardinals. Observe that in the proof of Proposition \ref{prop:FInA-step-up}, we used implicitly that there is a sequence $\langle f_\xi\mid 0<\xi<\gamma^+\rangle$ of surjections $f_\xi\colon\gamma\rightarrow\xi$ which cannot be proven to exist in $\mathrm{ZF}$. It is entirely possible $\omega_1$ is strongly MInA if the axiom of choice fails.

\begin{proposition}[$\mathrm{ZF}$]
If $\kappa$ is measurable then $\kappa$ is strongly MInA.
\end{proposition}

\begin{proof}
There are many ways to do this; here is one which fits into what we have done so far. First assume that choice holds. Let $U$ be a measure on $\kappa$ and let 
$$\langle M_n, j_{n, m}\mid n\leq m\leq \omega\rangle$$
be the iterated ultrapower of $M_0=V$ by $U$ of length $\omega$. The critical sequence $\langle j_{0, n}(\kappa)\mid n<\omega\rangle$ is forwards $j_{0,\omega}(\mathcal F)$-independent for any countable (or of size ${<}\kappa$) collection of functions $f\colon\kappa^{n_f}\rightarrow [\kappa]^{<\gamma}$. The reason is that for $n<\omega$, every $j_{0, m}(\kappa)$ with $m>n$ as well as $j_{0,\omega}(f)$ for $f\in \mathcal F$ is in the range of $j_{n, \omega}$, but $j_{0, n}(\kappa)$ is missing. The critical sequence is not an element of $M_\omega$, but we can use absoluteness of well-foundedness to show that there must be an infinite forward $\mathcal F$-independent set in $M_\omega$ nonetheless. So $j_{0,\omega}(\kappa)$ is $\gamma$-FInA in $M_\omega$, so $\kappa$ is $\gamma$-FInA in $V$ by elementarity of $j_{0,\omega}$. As $\gamma$-MInA is the same as $\gamma$-FInA, $\kappa$ is $\gamma$-MInA. We used the axiom of choice here in a few places, most notably in the form of \L os's theorem. \\
Now drop the assumption of the axiom of choice. Suppose $\gamma<\kappa$ and $\mathcal F$ is a countable collection of functions on $f\colon\kappa^{n_f}\rightarrow[\kappa]^{<\gamma}$. We can code $\mathcal F$ as a subset $A$ of $\kappa$, so that $\mathcal F\in L[A, U]$ is a model of $\mathrm{ZFC}$ in which $\kappa$ is measurable. Hence there is an infinite mutually $\mathcal F$-independent set in $L[A, U]$.
\end{proof}

It is well-known that $\omega_1$ can consistently with $\mathrm{ZF}$ be measurable, at least assuming that $\mathrm{ZFC}$+``there is a measurable cardinal" is consistent. In fact, this happens naturally if the axiom of determinacy holds. It follows that $\omega_1$ is consistently strongly MInA. Also note that our results on the equivalence between strongly losing and strongly MInA cardinals do not depend on the axiom of choice.

We challenge the reader who is sceptical of large cardinals to prove within $\mathrm{ZFC}$ that all $(\omega,\kappa,\omega)$-hat games are winning. 

\subsection{Strongly MInA cardinals in $L$}

 In $L$, one can compute exactly what the strongly MInA cardinals are in terms of the \textit{partition calculus}. This is a subject of Set Theory about infinite combinatorics; roughly it is about finding in a partition of a set an infinite subset of one of the parts with a nice structure. 

\begin{definition}
For cardinals $\kappa, \gamma$, the partition relation $\kappa\rightarrow (\omega)^{<\omega}_\gamma$ holds if for any function 
$$f\colon[\kappa]^{<\omega}\rightarrow \gamma$$
there is a countably infinite set $H\subseteq\kappa$ which is \textit{homogeneous} for $f$, i.e.~$f(s)=f(t)$ for all $s, t\in [H]^{<\omega}$ of the same size. We write $\kappa\rightarrow (\omega)^{<\omega}_{<\kappa}$ as shorthand for 
$$\forall\gamma<\kappa\ \kappa\rightarrow(\omega)^{<\omega}_\gamma.$$
\end{definition}
The following is similar to \cite[Theorem 2]{DevPar73} which is also attributed to Baumgartner there.
\begin{lemma}
In G\"odel's constructible universe $L$, a cardinal $\kappa$ is strongly MInA iff $\kappa\rightarrow(\omega)^{<\omega}_{<\kappa}$.
\end{lemma}

\begin{proof}
If $\kappa\rightarrow (\omega)^{<\omega}_{<\kappa}$ holds then we can deduce $\kappa$ to be strongly MInA without the restriction to $L$: Suppose $\gamma<\kappa$ and $\mathcal F=\{f_n\mid n<\omega\}$ is a countable collection of functions $f\colon \kappa^{m_n}\rightarrow \kappa^{<\gamma}$. By Lemma \ref{lem:FInA-MInA-equivalence}, it suffices to show that there is an infinite forwards $\mathcal F$-independent set. We may assume that $m_n=n$ for all $n$. Consider the partition
$$F\colon [\kappa]^{<\omega}\rightarrow \gamma+1$$
given by 
$$F(\{\alpha<\beta_1<\dots< \beta_n\})=\begin{cases}
\delta & \text{if }\alpha \text{ is the }\delta\text{-th element of }f_n(\beta_1,\dots\beta_n)\\
\gamma & \text{if }\alpha\notin f_n(\beta_1,\dots,\beta_n).
\end{cases}
$$
By assumption, there is a countably infinite set $H$ homogeneous for $F$. If $\alpha_0<\alpha_1<\beta_1<\dots<\beta_n$ are all in $H$ then 
$$\delta:=F(\{\alpha_0,\beta_1\dots,\beta_n\})=F(\{\alpha_1,\beta_1,\dots,\beta_n\})$$
and since $\alpha_0,\alpha_1$ cannot both occupy ``the same spot" in $f_n(\beta_1,\dots,\beta_n)$, we must have $\delta=\gamma$. It follows that $F$ takes only the value $\gamma$ on $[H]^{<\omega}$ and hence $H$ is forwards $\mathcal F$-independent.\\

For the other direction, suppose that $\kappa$ is strongly MInA and $\gamma<\kappa$. Consider the structure $(L_{\kappa^+}, \in)$, where $L_{\kappa^+}$ is the $H_{\kappa^+}$ as computed in $L$. Let $\mathcal F$ consist of all functions of the form
$$F\colon \kappa^n\rightarrow [\kappa]^\gamma,\  F(\alpha_1,\dots,\alpha_n)=\{f(\delta,\alpha_1,\dots,\alpha_n)\mid \delta\leq\gamma\}$$
where $f\colon\kappa^{n+1}\rightarrow \kappa$ is definable over $(L_{\kappa^+},\in)$, so $\mathcal F$ is countable. As strongly MInA cardinals are limit cardinals, $\kappa$ is $\gamma^+$-MInA, so there is an infinite mututally $\mathcal F$-independent set $X$.  For any $Y\subseteq X$ let $$S_Y\prec L_{\kappa^+}$$
be the elementary substructure generated by $Y\cup(\gamma+1)$, i.e.
$$S_Y=\{x\in L_{\kappa^+}\mid x\text{ is definable over }L_{\kappa^+}\text{ from parameters in }Y\cup(\gamma+1)\}.$$
By the Mostowski collapse theorem, for each $Y$ there is a transitive set $T_Y$ and an isomorphism $\pi_{Y}\colon T_Y\rightarrow S_Y$. A remarkable fact known as \textit{condensation}, which ties in to the minimality of $L$ among inner models, is that the structures $T_Y$ depend only on their ordinal height, i.e.~the least ordinal $\alpha_Y$ so that $\alpha_Y\notin T_Y$. As the ordinals are well-founded, the set
$$\{\alpha^Y\mid Y\text{ is an infinite subset of }X\}$$
contains a minimal element $\alpha^Y$. It follows that if $Z$ is any infinite subset of $Y$ then $T_Y=T_Z$, call this set $T$. Setting $\pi_{Z, Y}=\pi_Y^{-1}\circ\pi_Z$, we get an elementary embedding
$$\pi_{Z, Y}\colon  T\rightarrow T.$$
If $Z\subsetneq Y$ then $\pi_{Z, Y}$ cannot be the identity as any element of $Y\setminus Z$ is missing from the range of $\pi_Z$. This means that there must be a least ordinal $\mu\in T$ so that $\pi_{Z, Y}(\mu)\neq\mu$. Let $\pi=\pi_{Y, Z}$, $\tau=\pi_Y^{-1}$ and $\mu_0$ this least ordinal which is moved by $\pi$, note that $\mu_0>\gamma$. For $n<\omega$, set $\mu_n=\pi^n(\mu_0)$. An old argument of Silver \cite{Sil70} shows that in this scenario, $\{\mu_n\mid n<\omega\}$ is \textit{order indiscernible} over $(T,\in, (\alpha)_{\alpha\leq\gamma})$, meaning that 
\begin{equation*}\tag{I}\label{eq:order-indiscernible}
\varphi(\alpha, \mu_{i_1},\dots,\mu_{i_n})\longleftrightarrow\varphi(\alpha,\mu_{j_1},\dots,\mu_{j_n})
\end{equation*}
holds true in $(T,\in)$ for any $i_1<\dots< i_n<\omega$, $j_1<\dots<j_n<\omega$, $\alpha\leq\gamma$ and any first order $\in$-formula $\varphi$. If $\kappa\rightarrow(\omega)^{<\omega}_{\gamma}$ would not hold in $L$ then, due to the nature of $L$, there would be a counterexample $G\colon\kappa^{<\omega}\rightarrow\gamma$ definable over $L_{\kappa^+}$ from the parameter $\gamma$. It follows from (\ref{eq:order-indiscernible}) and the elementarity of $\pi_Y\colon T\rightarrow L_{\kappa^+}$ that $\{\pi_Y(\mu_n)\mid n<\omega\}$ is homogeneous for $G$, so there is no such \nolinebreak $G$.
\end{proof}

A cardinal $\kappa$ is a \textit{$\omega$-partition cardinal} iff there is some $\gamma$ so that $\kappa$ is least with $\kappa\rightarrow(\omega)^{<\omega}_\gamma$. These cardinals were introduced and studied by Drake \cite{Dra74}. The better known $\omega$-Erd\H{o}s cardinal is simply the least $\omega$-partition cardinal, if it exists. All $\omega$-partition cardinals are inaccessible and much more. Cardinals $\kappa$ with $\kappa\rightarrow(\omega)^{<\omega}_{<\kappa}$ are exactly the $\omega$-partition cardinals and limits of $\omega$-partition cardinals. It follows that, in $L$, the $\omega$-partition cardinals are exactly those which arise as losing thresholds for the $(\lambda,\underline{\hspace{5pt}}, \gamma)$-hat games for infinite $\lambda,\gamma$. So in $L$, when playing the $(\omega,\kappa,\gamma)$-game for infinite $\gamma$, the set of colours has to grow much much larger than $\gamma$ before the logicians start losing.

\end{document}